\documentclass[a4paper,reqno]{amsart}
\usepackage{microtype}
\usepackage{amssymb}
\usepackage{amsfonts}
\usepackage{geometry}
\usepackage{graphicx}

\newtheorem{theorem}{Theorem}
\theoremstyle{plain}

\newtheorem{corollary}{Corollary}

\newtheorem{definition}{Definition}
\newtheorem{example}{Example}

\newtheorem{lemma}{Lemma}

\numberwithin{equation}{section}
\numberwithin{theorem}{section}  
\numberwithin{proposition}{section}  
\numberwithin{lemma}{section}  
\numberwithin{corollary}{section}

\input{tcilatex}

\begin{document}
\title[A Cylinder Moving in a Kinetic Gas]{Approach to Equilibrium of a Body
Colliding Specularly and Diffusely with a Sea of Particles}
\author{Xuwen Chen}
\address{Department of Mathematics, Brown University, Providence, RI 02912}
\email{chenxuwen@math.brown.edu}
\urladdr{http://www.math.brown.edu/\symbol{126}chenxuwen/}
\author{Walter Strauss}
\address{Department of Mathematics and Lefschetz Center for Dynamical
Systems, Brown University, Providence, RI 02912}
\email{wstrauss@math.brown.edu}
\urladdr{http://www.math.brown.edu/\symbol{126}wstrauss/}
\date{May 10, 2013}

\begin{abstract}
We consider a rigid body acted upon by two forces, a constant force and the
collective force of interaction with a continuum of particles. We assume
that some of the particles that collide with the body reflect elastically
(specularly), while others reflect probabilistically with some probablility
distribution $K$. We find that the rate of approach of the body to
equilibrium is $O(t^{-3-p})$ in three dimensions where $p$ can take any
value from 0 to 2, depending on $K$.
\end{abstract}

\maketitle
\tableofcontents

\section{Introduction}

The problem that we are considering has a free boundary, the location of the
body. The other unknown is the configuration of the particles. The particles
may collide with the body elastically or inelastically. Boundary
interactions in kinetic theory are very poorly understood, even when the
boundaries are fixed. Free boundaries are even more difficult. For this
reason we have chosen to consider only the \textit{simplest} problem of this
type, namely, we assume the particles are identical and are rarefied, that
is, do not interact among themselves but only with the body. We assume that
the whole system, consisting of the body and the particles, starts out
rather close to an equilibrium state.

We consider classical particles that are extremely numerous. While one could
consider modeling them as a fluid, we instead model them as a continuum like
in kinetic (Boltzmann, Vlasov) theory \cite{Glassey} but without any
self-interaction. Our focus is on the interaction of the particles with the
body at its boundary. In typical physical scenarios this interaction is
poorly understood. For instance, the boundary may be so rough that a
particle may reflect from it in an essentially random way. There could even
be some kind of physical or chemical reaction between the particle and the
molecules of the body.

The present paper treats a problem similar to the series of remarkable
papers \cite{Ital1, Ital2, Ital3} and uses similar methods. In each of these
papers the initial velocity $V_0$ of the body is near its terminal
(equilibrium) velocity $V_\infty$ and the body moves in only one spatial
direction (to the right). In \cite{Ital1} and \cite{Ital2} all the
collisions are purely specular. The body's initial velocity satisfies $%
V_0<V_\infty$ in \cite{Ital1}, while $V_0>V_\infty$ in \cite{Ital2}. The
latter case is significantly different. In \cite{Ital3} the collisions are
purely diffusive with the collision kernel $K(\mathbf{v},\mathbf{u}) =
C|u_x| e^{-\beta |\mathbf{v}|^2}$ where $C$ and $\beta$ are constants. This
kernel implies that all the colliding particles are emitted with the same
Maxwellian distribution. In the present paper we generalize the boundary
behavior to permit a mixture of specular and diffusive reflections. The
diffusion part is much more general than in \cite{Ital3}.

In \cite{Ital1} and \cite{Ital2} the rate of approach of the velocity to
equilibrium is $O(t^{-d-2})$ in $d$ spatial dimensions. At first glance it
is somewhat surprising that the rate is slower than exponential. This
relatively slow rate is due to some particles colliding with the body
multiple times over long time periods, which produces a frictional effect on
the body that may be called a \textit{long tail memory}. In \cite{Ital3} the
rate is slower, namely, $O(t^{-d-1})$, because the number of collisions is
greater due to the diffuse reflections. In the present paper we find various
rates of approach depending on the specific law of reflection. We find the
rate $O(t^{-d-p})$ where $p$ can take any value from 0 to 2.

In physically realistic situations, many more effects must be included, such
as thermal effects, collisions among the particles themselves, or
electromagnetic effects. In a plasma the particles are usually modeled
kinetically, as for instance the reentry of a space vehicle into the
atmosphere. Another way to model particles that interact with a body would
be to treat them as a classical fluid. For a general discussion on
fluid-structure interaction, see \cite{Piston1}. Somewhat related to this
paper is the piston problem, where the body is a piston moving back and
forth in a finite channel \cite{Piston2} and naturally reaching an
equilibrium state. However, the piston problem is different primarily
because the particles reflect at the ends of the channel and collide an
infinite number of times, rather than scattering to infinity. More relevant
to this paper are the numerical computations in \cite{ATC, TA}, which
corroborate the power-law asymptotic behavior for the diffuse boundary
conditions of \cite{Ital3}. In \cite{Cav} a general convex body, moving
horizontally, is considered, and the results are similar to \cite{Ital1}.

To be specific, here we consider the following problem. The body is a
cylinder $\Omega (t)\subset \mathbb{R}^{d}$. We write $\mathbf{x}%
=(x,x_{\perp }),\ x_{\perp }\in \mathbb{R}^{d-1}$. The cylinder is parallel
to the $x$-axis and the body is constrained to move only in the $x$
direction with velocity $V(t)$. There is a constant horizontal force $E>0$
acting on the body, as well as the horizontal force $F(t)$ due to all the
colliding particles at time $t$. Thus 
\begin{equation*}
\frac{dX}{dt}=V(t),\quad \frac{dV}{dt}=E-F(t),
\end{equation*}%
In the fictitious situation that none of the particles collide more than
once with the body, their collective force on the body is denoted as $%
F_{0}(V)$. (See Lemma \ref{fictitious lemma}.) Then the equilibrium velocity
would be $V_{\infty }$, where $F_{0}(V_{\infty })=E$.

We write the velocity of a particle as $\mathbf{v}=(v_{x},v_{\perp })$,
where $v_{x}=\mathbf{v}\cdot \mathbf{i}$ is the horizontal component and $%
v_{\perp }\in \mathbb{R}^{d-1}$. The particle distribution, denoted by $f(t,%
\mathbf{x},\mathbf{v})$, satisfies $\partial _{t}f+\mathbf{v}\cdot \nabla _{%
\mathbf{x}}f=0$ in $\Omega ^{c}(t)$. We assume the initial velocity $f(0,%
\mathbf{x},\mathbf{v})=f_{0}(\mathbf{v})$ depends only on $\mathbf{v}$ and
is even in $v_{x}$. We also denote the densities before and after a
collision with the body by $f_{\pm }(t,\mathbf{x},\mathbf{v})=\lim_{\epsilon
\rightarrow 0^{+}}f(t\pm \epsilon ,\mathbf{x}\pm \epsilon \mathbf{v},\mathbf{%
v})$. The assumed law of reflection at the two ends of the cylinder is 
\begin{equation*}
f_{+}(t,\mathbf{x};\mathbf{v})=\alpha f_{-}(t,\mathbf{x},2V(t)-v_{x},v_{\bot
})+(1-\alpha )\int_{(u_{x}-V\left( t\right) )(v_{x}-V\left( t\right) )\leq
0}K\left( \mathbf{v}-\mathbf{i}V\left( t\right) ;\mathbf{u}-\mathbf{i}%
V\left( t\right) \right) f_{-}(t,\mathbf{x};\mathbf{u})d\mathbf{u},
\end{equation*}%
where $\mathbf{i}$ is the unit vector in the $x$-direction and $\alpha \in %
\left[ 0,1\right) $. The collision kernel $K$ is assumed to satisfy the
conservation of mass condition \eqref{kernelmass}. Furthermore, $K$ and the
initial density $f_{0}$ satisfy Assumptions A1-A5 in Section 3. Among these
conditions are 
\begin{equation*}
K(\mathbf{v,u})=k(v_{x},u_{x})b(v_{\bot }),\quad c\left\vert
u_{x}\right\vert ^{p}\leqslant \ \int_{v_{x}\geq 0}\ v_{x}^{2}\
k(v_{x},u_{x})\ dv_{x}\ \leqslant C\left\vert u_{x}\right\vert ^{p}
\end{equation*}%
for some constants $c,C,p$ and some function $b(v_{\perp })$ where $0\leq
p\leq 2$. A symmetry assumption implies that the net force on the lateral
boundary vanishes (Lemma \ref{lateral force}).

\begin{theorem}
\label{ThExistence} Given a collision kernel and the initial data $f_{0}$ as
above. If $\gamma =V_{\infty }-V_{0}$ is sufficiently small and positive,
then there exists a solution $(V(t),f(t,x,v))$ of our problem in the
following sense. $V\in C^{1}(\mathbb{R})$ and $f_{\pm }\in L^{\infty }$ for $%
t\in \lbrack 0,\infty ),x\in \partial \Omega (t),v\in \mathbb{R}^{3}$, where
the force $F(t)$ on the cylinder is given by \eqref{force:used formula} 
and the pair of functions $f_{\pm }(t,x,v)$ are (almost everywhere) defined
explicitly in terms of $V(t)$ and $f_{0}(x,v)$.
\end{theorem}

Uniqueness is an open problem, as in \cite{Ital1, Ital2, Ital3}.

\begin{theorem}
\label{ThEvery} Every solution of the problem (in the sense stated above)
satisfies the estimates 
\begin{equation}
\gamma e^{-B_{\infty }t}+\frac{c\gamma ^{p+1}\chi _{(2t_{0},\infty )}\left(
t\right) }{t^{d+p}}\leqslant V_{\infty }-V(t)\leqslant \gamma e^{-B_{0}t}+%
\frac{C\gamma ^{p+1}}{(1+t)^{d+p}},
\end{equation}%
where $B_{0}=\min_{V\in \left[ V_{0},V_{\infty }\right] }F_{0}^{\prime
}(V)>0\ \ $and$\ \ B_{\infty }=\max_{V\in \left[ V_{0},V_{\infty }\right]
}F_{0}^{\prime }(V)<\infty $, for some $t_{0}$ depending on $\gamma $ and
some positive constants $c,C$.
\end{theorem}

In Section 2 we derive the basic formulas for the total force on a body due
to its interaction with the particles. This is done directly from basic
principles in a more organized way than in \cite{Ital3}. In terms of the
boundary conditions the total force is given in Lemma \ref{total force}. We
assume that particles are not created or annihilated at the boundary
(conservation of mass) (Lemmas \ref{mass} and \ref{lateral mass}). Under
either of two lateral boundary conditions, the force on the lateral side of
the body can be ignored (Lemmas \ref{lateral force} and \ref{alt lateral
force}).

Section 3 introduces a family $\mathcal{W}$ of possible body motions $W$, in
terms of two functions $g(t)$ and $h(t)$ which we determine later. We write
the force due to the possible motion $W$ as $F(t)=F_{0}(t)+R_{W}(t)$, where $%
R_{W}(t)$ is the force due to the collisions occurring before time $t$
(\textquotedblleft precollisions") if the body were to move with velocity $%
W(\cdot )$. Then $W$ generates a new possible motion $V_{W}$ by the equation 
\begin{equation}
\frac{dV_{W}}{dt}=E-F_{0}(t)-R_{W}(t).  \label{iteration equation}
\end{equation}%
The goal is to prove that the mapping $W\rightarrow V_{W}$ has a fixed
point. At the end of the section we list all the assumptions on the
collision kernel $K(\mathbf{v},\mathbf{u})$ and the initial state $f_{0}(%
\mathbf{v})$ of the particles, stated in as general a form as feasible.

In Section 4 we provide several examples of collision kernels for the two
ends of the cylinder. Example 1 is the same gaussian collision law $%
k(v_{x},u_{x})=C|u_{x}|\exp (-\beta v_{x}^{2})$ considered in \cite{Ital3},
where the particles comprise a perfect gas in thermal equilibrium. In that
paper the exponent $p=1$ and the authors assume that $V_{\infty }$ is
sufficiently large without specifying how large. We provide an explicit
condition \eqref{Vinfinity
large} on the size of $V_{\infty }$. We also provide an alternative
condition \eqref{beta small} on the shape of the gaussian that is
independent of $V_{\infty }$.

Example 2 is more interesting. The kernel is $k(v_x,u_x) = C\exp(-v_x^2 /
|u_x|)$ and the value of $p$ is $\frac32$. This means that, for a particle
colliding at an incoming velocity $u_x$ close to that of the body $V(t)$,
its outgoing velocity $v_x$ upon reflection is given by a narrow gaussian
and so is likely to be not very changed. Thus the particles that are almost
grazing are deviated only slightly. On the other hand, if $u_x$ is quite far
from $V(t)$, that is if the collision is more fierce, the particle's
velocity upon reflection is given by a very wide gaussian and so is likely
to take almost any value. It seems that this may be a more realistic
scattering scenario than the one in Example 1.

Example 3 proposes a family of kernels that generalize both previous
examples, permitting any $p\in[0,2]$. Example 4 shows that there is no
requirement that the kernel is an exponential; all that is needed is some
polynomial decay.

Section 5 is devoted to our main estimates on how the particles collectively
generate a force on the body. Because $E>0$, there is a difference between
the left and the right sides. The bounds on the force employ mainly the
first precollisions. The most important conclusion is that $R_W(t)\ge0$. 
The estimates are summarized in Lemma \ref{Corollary:UpperAndLowerBoundOfR}.

In Section 6 we apply the estimates of the force $R_{W}$ to the body's
motion using \eqref{iteration equation}. The functions $g(t)$ and $h(t)$ in
the definition of $\mathcal{W}$ can then be chosen. Finally in Section 7 we
deduce the existence theorem, Theorem \ref{ThExistence}, by a fixed point
argument that iterates the upper bound of $R_{W}(t)$, thus taking account of
all the precollisions. The asymptotic theorem, Theorem \ref{ThEvery}, valid
for any solution, follows easily.

\subsection{Acknowledgements}

We thank Kazuo Aoki for introducing the second author to the previous work
on this subject and for very generous discussions, Andong He for general
discussions, Constantine Dafermos for his advice on mechanics, and Mario
Pulvirenti for elaborating the proof of convergence in \cite{Ital3}.


\section{Force on the Body}

\subsection{Force and Flux}

In this section, we derive the formula of the force on the body assuming
conservation of mass. Let $\Omega (t)$ be the moving body at time $t$, $%
\Omega ^{c}(t)$ its complement (where the particles are located), and $%
\mathbf{v}=(v_{x},v_{\bot })$ the velocity of a particle. We normalize the
body's mass and the particle density to be 1. Then the mass (the total
number of particles) is 
\begin{equation*}
M(t) = \int_{\Omega ^{c}(t)}d\mathbf{x}\int_{{\mathbb{R}}^{3}}\ d\mathbf{v}%
f(t,\mathbf{x},\mathbf{v}).
\end{equation*}%
Conservation of mass means that the flux across the boundary vanishes; that
is, 
\begin{equation}
\int_{\partial \Omega (t)}dS_{\mathbf{x}}\int_{{\mathbb{R}}^{3}}d\mathbf{v}%
\left[ \left( \mathbf{v}-\mathbf{V}(t)\right) \cdot \mathbf{n}\right] f(t,%
\mathbf{x},\mathbf{v})=0,  \label{relation:conservation of particles}
\end{equation}%
where $\partial \Omega (t)$ is the boundary of $\Omega (t)$, $\mathbf{n} = 
\mathbf{n}_\mathbf{x}$ is the outward normal on $\partial \Omega (t)$, and $%
\mathbf{V}=(V_{x},V_{\bot })$ is the velocity of the body. In Section \ref%
{Sec:MassConservation} we will find boundary conditions so that %
\eqref{relation:conservation of particles} is valid.

\begin{lemma}
Assuming conservation of mass (\ref{relation:conservation of particles}),
the horizontal component of the force on the body is given by the formula 
\begin{equation}
F(t)=\int_{\partial \Omega (t)}dS_{\mathbf{x}}\int_{{\mathbb{R}}^{3}}d%
\mathbf{v}\left[ \left( \mathbf{v}-\mathbf{V}(t)\right) \cdot \mathbf{n}%
\right] \left( v_{x}-V_{x}\right) f(t,\mathbf{x},\mathbf{v}).
\label{equation:derivation of F(t)}
\end{equation}
\end{lemma}

\begin{proof}
The change of horizontal momentum of the gas and the moving solid together
is given by%
\begin{equation*}
\frac{d}{dt} \left\{ \int_{\Omega ^{c}(t)}d\mathbf{x}\int_{{\mathbb{R}}^{3}}d%
\mathbf{v}\ v_{x}f(t,\mathbf{x},\mathbf{v})+{V}\left( t\right) \right\} =E
\end{equation*}%
Hence $F(t)$ is 
\begin{eqnarray*}
F(t) &=&\frac{d}{dt}\int_{\Omega ^{c}(t)}d\mathbf{x}\int_{{\mathbb{R}}^{3}}d%
\mathbf{v}v_{x}f(t,\mathbf{x},\mathbf{v}) \\
&=&-\int_{\partial \Omega (t)}dS_{\mathbf{x}}\int_{{\mathbb{R}}^{3}}d\mathbf{%
v}(\mathbf{V}(t)\cdot \mathbf{n})v_{x}f(t,\mathbf{x},\mathbf{v}%
)+\int_{\Omega ^{c}(t)}d\mathbf{x}\int_{{\mathbb{R}}^{3}}d\mathbf{v}%
v_{x}\partial _{t}f(t,\mathbf{x},\mathbf{v}) \\
&=&-\int_{\partial \Omega (t)}dS_{\mathbf{x}}\int_{{\mathbb{R}}^{3}}d\mathbf{%
v}(\mathbf{V}(t)\cdot \mathbf{n})v_{x}f(t,\mathbf{x},\mathbf{v}%
)-\int_{\Omega ^{c}(t)}d\mathbf{x}\int_{{\mathbb{R}}^{3}}d\mathbf{v}%
v_{x}\left( \mathbf{v}\cdot \nabla _{\mathbf{x}}f\right) \\
&=&-\int_{\partial \Omega (t)}dS_{\mathbf{x}}\int_{{\mathbb{R}}^{3}}d\mathbf{%
v}(\mathbf{V}(t)\cdot \mathbf{n})v_{x}f(t,\mathbf{x},\mathbf{v}%
)+\int_{\partial \Omega (t)}dS_{\mathbf{x}}\int_{{\mathbb{R}}^{3}}\left( 
\mathbf{v}\cdot \mathbf{n}\right) v_{x}f(t,\mathbf{x},\mathbf{v})d\mathbf{v}
\\
&=&\int_{\partial \Omega (t)}dS_{\mathbf{x}}\int_{{\mathbb{R}}^{3}}d\mathbf{v%
}\left[ \left( \mathbf{v}-\mathbf{V}(t)\right) \cdot \mathbf{n}\right]
v_{x}f(t,\mathbf{x},\mathbf{v}).
\end{eqnarray*}%
Via conservation of mass (\ref{relation:conservation of particles}), we can
write 
\begin{eqnarray*}
F(t) &=&\frac{d}{dt}\int_{\Omega ^{c}(t)}d\mathbf{x}\int_{{\mathbb{R}}^{3}}d%
\mathbf{v}v_{x}f(t,\mathbf{x},\mathbf{v})-V_{x}\frac{d}{dt}\int_{\Omega
^{c}(t)}d\mathbf{x}\int_{{\mathbb{R}}^{3}}d\mathbf{v}f(t,\mathbf{x},\mathbf{v%
}) \\
&=&\int_{\partial \Omega (t)}dS_{\mathbf{x}}\int_{{\mathbb{R}}^{3}}d\mathbf{v%
}\left[ \left( \mathbf{v}-\mathbf{V}(t)\right) \cdot \mathbf{n}\right]
\left( v_{x}-V_{x}\right) f(t,\mathbf{x},\mathbf{v}).
\end{eqnarray*}
\end{proof}

\begin{lemma}
If we specialize $\Omega (t)$ to a cylinder centered at $(X(t),0,0)$ with
its circular base perpendicular to the x-axis and moving only horizontally,
i.e. $\mathbf{V}=(V,0,0)$, then the horizontal force $F$ is given by%
\begin{equation}
F(t)=F_{L,R}(t)+F_{S}(t),  \label{force:used formula}
\end{equation}%
where $F_{L,R}(t)$ is the contribution from both $\partial \Omega _{L}(t)$
and $\partial \Omega _{R}(t)$, the left and right ends of the cylinder, and $%
F_{S}(t)$ is the contribution from $\partial \Omega _{S}(t)$, the lateral
side of the cylinder. Written explicitly in terms of the incident and
reflected particles, they are%
\begin{eqnarray*}
F_{L,R}(t) &=&\int_{\partial \Omega _{R}(t)}dS_{\mathbf{x}%
}\int_{v_{x}\leqslant V(t)}d\mathbf{v}[v_{x}-V(t)]^{2}f_{-}\left( t,\mathbf{x%
},\mathbf{v}\right) \\
&&+\int_{\partial \Omega _{R}(t)}dS_{\mathbf{x}}\int_{v_{x}\geqslant V(t)}d%
\mathbf{v}[v_{x}-V(t)]^{2}f_{+}\left( t,\mathbf{x},\mathbf{v}\right) \\
&&-\int_{\partial \Omega _{L}(t)}dS_{\mathbf{x}}\int_{v_{x}\geqslant V(t)}d%
\mathbf{v}[v_{x}-V(t)]^{2}f_{-}\left( t,\mathbf{x},\mathbf{v}\right) \\
&&-\int_{\partial \Omega _{L}(t)}dS_{\mathbf{x}}\int_{v_{x}\leqslant V(t)}d%
\mathbf{v}[v_{x}-V(t)]^{2}f_{+}\left( t,\mathbf{x},\mathbf{v}\right)
\end{eqnarray*}%
and%
\begin{eqnarray}
F_{S}(t) &=&\int_{\partial \Omega _{S}(t)}dS_{\mathbf{x}}\int_{\mathbf{v}%
\cdot \mathbf{n\leqslant }0}d\mathbf{v}\left( \mathbf{v}\cdot \mathbf{n}%
\right) \left( v_{x}-V\right) f_{-}(t,\mathbf{x},\mathbf{v})
\label{force:lateral} \\
&&+\int_{\partial \Omega _{S}(t)}dS_{\mathbf{x}}\int_{\mathbf{v}\cdot 
\mathbf{n\geqslant }0}d\mathbf{v}\left( \mathbf{v}\cdot \mathbf{n}\right)
\left( v_{x}-V\right) f_{+}(t,\mathbf{x},\mathbf{v}).  \notag
\end{eqnarray}
\end{lemma}

\begin{proof}
At the right side of the cylinder, $\mathbf{n}=(1,0,0)$ and $\mathbf{V}%
=(V,0,0)$. So for the right end, we get from 
\eqref{equation:derivation
of F(t)} the term%
\begin{eqnarray*}
&&\int_{\partial \Omega _{R}(t)}dS_{\mathbf{x}}\int_{{\mathbb{R}}^{3}}d%
\mathbf{v}\left[ \left( \mathbf{v}-\mathbf{V}(t)\right) \cdot \mathbf{n}%
\right] \left( v_{x}-V\right) f(t,\mathbf{x},\mathbf{v}) \\
&=&\int_{\partial \Omega _{R}(t)}dS\int d\mathbf{v}[v_{x}-V(t)]^{2}f(t,%
\mathbf{x},\mathbf{v}).
\end{eqnarray*}%
Splitting the above integral into its absorbed (incident) and emitted
(scattered) parts as%
\begin{equation*}
f\left( t,\mathbf{x},\mathbf{v}\right) =f_{-}\left( t,\mathbf{x},\mathbf{v}%
\right) \chi \left( \left\{ v_{x}\leqslant V(t)\right\} \right) +f_{+}\left(
t,\mathbf{x},\mathbf{v}\right) \chi \left( \left\{ v_{x}>V(t)\right\}
\right) ,
\end{equation*}%
we have%
\begin{eqnarray}
&&\int_{\partial \Omega _{R}(t)}dS_{\mathbf{x}}\int d\mathbf{v}%
[v_{x}-V(t)]^{2}f(t,\mathbf{x},\mathbf{v})  \label{force:right} \\
&=&\int_{\partial \Omega _{R}(t)}dS_{\mathbf{x}}\int_{v_{x}\leqslant V(t)}d%
\mathbf{v}[v_{x}-V(t)]^{2}f_{-}\left( t,\mathbf{x},\mathbf{v}\right)  \notag
\\
&&+\int_{\partial \Omega _{R}(t)}dS_{\mathbf{x}}\int_{v_{x}\geqslant V(t)}d%
\mathbf{v}[v_{x}-V(t)]^{2}f_{+}\left( t,\mathbf{x},\mathbf{v}\right) . 
\notag
\end{eqnarray}%
Similarly, at the left end of the cylinder we have 
\begin{eqnarray}
&&\int_{\partial \Omega _{L}(t)}dS_{\mathbf{x}}\int_{{\mathbb{R}}^{3}}d%
\mathbf{v}\left[ \left( \mathbf{v}-\mathbf{V}(t)\right) \cdot \mathbf{n}%
\right] \left( v_{x}-V\right) f(t,\mathbf{x},\mathbf{v})  \label{force:left}
\\
&=&-\int_{\partial \Omega _{L}(t)}dS\int d\mathbf{v}[v_{x}-V(t)]^{2}f(t,%
\mathbf{x},\mathbf{v})  \notag \\
&=&-\int_{\partial \Omega _{L}(t)}dS_{\mathbf{x}}\int_{v_{x}\geqslant V(t)}d%
\mathbf{v}[v_{x}-V(t)]^{2}f_{-}\left( t,\mathbf{x},\mathbf{v}\right)  \notag
\\
&&-\int_{\partial \Omega _{L}(t)}dS_{\mathbf{x}}\int_{v_{x}\leqslant V(t)}d%
\mathbf{v}[v_{x}-V(t)]^{2}f_{+}\left( t,\mathbf{x},\mathbf{v}\right) . 
\notag
\end{eqnarray}%
Adding (\ref{force:right}) and (\ref{force:left}) gives $F_{L,R}(t).$

At the lateral boundary we have%
\begin{eqnarray*}
F_{S}(t) &=&\int_{\partial \Omega _{S}(t)}dS_{\mathbf{x}}\int_{{\mathbb{R}}%
^{3}}d\mathbf{v}\left[ \left( \mathbf{v}-\mathbf{V}(t)\right) \cdot \mathbf{n%
}\right] \left( v_{x}-V\right) f(t,\mathbf{x},\mathbf{v}) \\
&=&\int_{\partial \Omega _{S}(t)}dS_{\mathbf{x}}\int_{{\mathbb{R}}^{3}}d%
\mathbf{v}\left( \mathbf{v}\cdot \mathbf{n}\right) \left( v_{x}-V\right) f(t,%
\mathbf{x},\mathbf{v}).
\end{eqnarray*}%
Here the sign of $\mathbf{v}\cdot \mathbf{n}$ indicates the incident and
scattered particles, namely, 
\begin{equation*}
f\left( t,\mathbf{x},\mathbf{v}\right) =f_{-}\left( t,\mathbf{x},\mathbf{v}%
\right) \chi \left( \left\{ \mathbf{v}\cdot \mathbf{n}\leqslant 0\right\}
\right) +f_{+}\left( t,\mathbf{x},\mathbf{v}\right) \chi \left( \left\{ 
\mathbf{v}\cdot \mathbf{n}>0\right\} \right) ,
\end{equation*}%
so that 
\begin{eqnarray*}
F_{S}(t) &=&\int_{\partial \Omega _{S}(t)}dS_{\mathbf{x}}\int_{\mathbf{v}%
\cdot \mathbf{n\leqslant }0}d\mathbf{v}\left( \mathbf{v}\cdot \mathbf{n}%
\right) \left( v_{x}-V\right) f_{-}(t,\mathbf{x},\mathbf{v}) \\
&&+\int_{\partial \Omega _{S}(t)}dS_{\mathbf{x}}\int_{\mathbf{v}\cdot 
\mathbf{n\geqslant }0}d\mathbf{v}\left( \mathbf{v}\cdot \mathbf{n}\right)
\left( v_{x}-V\right) f_{+}(t,\mathbf{x},\mathbf{v}),
\end{eqnarray*}%
which is exactly (\ref{force:lateral}).
\end{proof}

\subsection{Boundary Conditions\label{Sec:MassConservation}}

In this section, we introduce boundary conditions for the scattering of the
particles which satisfy conservation of mass (\ref{relation:conservation of
particles}). There are three boundaries we are considering, namely, $%
\partial \Omega _{R}(t),$ $\partial \Omega _{L}(t),\partial \Omega _{S}(t).$
We will consider first $\partial \Omega _{R}(t)$ and $\partial \Omega
_{L}(t) $, then $\partial \Omega _{S}(t).$

\subsubsection{Boundary Conditions at the Two Ends}

On $\partial \Omega _{R}(t)$, the right circular base of the cylinder, for $%
v_{x}\geqslant V(t)$, we assume the boundary condition 
\begin{equation}
f_{+}(t,\mathbf{x};\mathbf{v})=\alpha f_{-}(t,\mathbf{x},2V(t)-v_{x},v_{\bot
})+(1-\alpha )\int_{u_{x}\leqslant V\left( t\right) }K\left( \mathbf{v}-%
\mathbf{i}V\left( t\right) ;\mathbf{u}-\mathbf{i}V\left( t\right) \right)
f_{-}(t,\mathbf{x};\mathbf{u})d\mathbf{u},  \label{boundary condition: Right}
\end{equation}%
Similarly, on $\partial \Omega _{L}(t)$, the left circular base of the
cylinder, for $v_{x}\leqslant V(t),$ we assume the boundary condition 
\begin{equation}
f_{+}(t,\mathbf{x};\mathbf{v})=\alpha f_{-}(t,\mathbf{x};2V(t)-v_{x},v_{\bot
})+(1-\alpha )\int_{u_{x}\geqslant V\left( t\right) }K\left( \mathbf{v}-%
\mathbf{i}V\left( t\right) ;\mathbf{u}-\mathbf{i}V\left( t\right) \right)
f_{-}(t,\mathbf{x};\mathbf{u})d\mathbf{u}.  \label{boundary condition: Left}
\end{equation}%
We assume $\alpha \in \left[ 0,1\right) $ since we are interested in the
mixed boundary condition, part specular and part diffusing. We would like to
have the same law of reflection on both circular ends of the cylinder so we
assume that the kernel $K\left( \mathbf{v},\mathbf{u}\right)$ is nonnegative
and is even in both $u_x$ and $v_x$ separately, namely, 
\begin{equation*}
K\left( v_{x},v_{\bot };u_{x},u_{\bot }\right) = K(-v_x,v_\perp;
u_x,u_\perp) = K(v_x,v_\perp; -u_x,u_\perp), \quad \forall u=(u_x,u_\perp),
v=(v_x,v_\perp) \in\mathbb{R}^3.
\end{equation*}

\begin{lemma}
\label{mass} If the collision kernel $K(\mathbf{v,u})$ satisfies 
\begin{equation}
\int_{v_{x}\geqslant 0}v_{x}K(\mathbf{v,u})d\mathbf{v}=\left\vert
u_{x}\right\vert ,  \label{kernelmass}
\end{equation}%
then across both ends of the cylinder the mass is conserved. This means that 
\begin{equation}  \label{flux}
\int_{\partial \Omega _{R}(t)}dS_{\mathbf{x}}\int_{{\mathbb{R}}^{3}}d\mathbf{%
v}\left[ \left( \mathbf{v}-\mathbf{V}(t)\right) \cdot \mathbf{n}\right] f(t,%
\mathbf{x},\mathbf{v}) = \int_{\partial \Omega _{L}(t)}dS_{\mathbf{x}}\int_{{%
\mathbb{R}}^{3}}d\mathbf{v}\left[ \left( \mathbf{v}-\mathbf{V}(t)\right)
\cdot \mathbf{n}\right] f(t,\mathbf{x},\mathbf{v})=0.
\end{equation}
\end{lemma}

\begin{proof}
Because it is very well known that specular reflection preserves mass, we
only consider the diffusing term in (\ref{boundary condition: Right}) and (%
\ref{boundary condition: Left}). The proof is the same on the left and the
right. Consider the right end. Splitting the integral \eqref{flux} at the
right end into its absorbed and emitted parts, we have 
\begin{equation*}
f\left( t,\mathbf{x},\mathbf{v}\right) =f_{-}\left( t,\mathbf{x},\mathbf{v}%
\right) \chi \left( \left\{ v_{x}\leqslant V_{x}(t)\right\} \right)
+f_{+}\left( t,\mathbf{x},\mathbf{v}\right) \chi \left( \left\{
v_{x}>V_{x}(t)\right\} \right) ,
\end{equation*}%
and 
\begin{eqnarray*}
&&\int_{\partial \Omega _{R}(t)}dS_{\mathbf{x}}\int_{{\mathbb{R}}^{3}}d%
\mathbf{v}\left[ \left( \mathbf{v}-\mathbf{V}(t)\right) \cdot \mathbf{n}%
\right] f(t,\mathbf{x},\mathbf{v}) \\
&=&\int_{\partial \Omega _{R}(t)}dS_{\mathbf{x}}\int_{v_{x}\leqslant
V_{x}(t)}d\mathbf{v}\left[ \left( v_{x}-V_{x}(t)\right) \right] f_{-}(t,%
\mathbf{x},\mathbf{v}) \\
&&+\int_{\partial \Omega _{R}(t)}dS_{\mathbf{x}}\int_{v_{x}>V_{x}(t)}d%
\mathbf{v}\left[ \left( v_{x}-V_{x}(t)\right) \right] f_{+}(t,\mathbf{x},%
\mathbf{v}) \\
&=&I+II.
\end{eqnarray*}%
Plugging the boundary condition into $II$, we have%
\begin{eqnarray*}
II &=&\int_{\partial \Omega _{R}(t)}dS_{\mathbf{x}}\int_{v_{x}\geqslant
V_{x}(t)}d\mathbf{v}\left[ \left( v_{x}-V_{x}(t)\right) \right]
\int_{u_{x}\leqslant V_{x}(t)}K(\mathbf{v-i}V_{x}(t),\mathbf{u}-\mathbf{i}%
V_{x}(t))f_{-}(t,\mathbf{x},\mathbf{u}) \\
&=&\int_{\partial \Omega _{R}(t)}dS_{\mathbf{x}}\int_{u_{x}\leqslant
V_{x}(t)}d\mathbf{u}f_{-}(t,\mathbf{x},\mathbf{u})\int_{v_{x}\geqslant
V_{x}(t)}d\mathbf{v}\left[ \left( v_{x}-V_{x}(t)\right) \right] K(\mathbf{v-i%
}V_{x}(t),\mathbf{u}-\mathbf{i}V_{x}(t)).
\end{eqnarray*}%
By \eqref{kernelmass} we have 
\begin{equation*}
\int_{v_{x}\geqslant V_{x}(t)}d\mathbf{v}\left[ \left( v_{x}-V_{x}(t)\right) %
\right] K(\mathbf{v-i}V_{x}(t),\mathbf{u}-\mathbf{i}V_{x}(t))=-\left(
u_{x}-V_{x}(t)\right)
\end{equation*}%
for $u_{x}\leqslant V_{x}.$ Hence 
\begin{equation*}
II=-\int_{\partial \Omega _{R}(t)}dS_{\mathbf{x}}\int_{v_{x}\leqslant
V_{x}(t)}d\mathbf{v}\left[ \left( v_{x}-V_{x}(t)\right) \right] f_{-}(t,%
\mathbf{x},\mathbf{v})=-I.
\end{equation*}%
Thus we conclude that 
\begin{equation*}
\int_{\partial \Omega _{R}(t)}dS_{\mathbf{x}}\int_{{\mathbb{R}}^{3}}d\mathbf{%
v}\left[ \left( \mathbf{v}-\mathbf{V}(t)\right) \cdot \mathbf{n}\right] f(t,%
\mathbf{x},\mathbf{v})=0.
\end{equation*}
\end{proof}


\subsubsection{Boundary Conditions on the Lateral Boundary}

As with the ends, we impose a linear combination of specular and diffusing
boundary conditions on $\partial \Omega _{S}$. Let $\mathbf{n}_{\mathbf{x}}$
be the outward normal and let and $\mathbf{T}_{\mathbf{x}}$ be the circular
tangential direction at $\mathbf{x}\in \partial \Omega _{S}$. We assume on $%
\partial \Omega _{S}$ the boundary condition 
\begin{equation}
f_{+}(t,\mathbf{x};\mathbf{v})=\alpha _{S}f_{-}(t,\mathbf{x};v_{x},\mathbf{%
v\cdot T}_{\mathbf{x}},-\mathbf{v\cdot n}_{\mathbf{x}})+(1-\alpha _{S})\int_{%
\mathbf{u}\cdot \mathbf{n}_{\mathbf{x}}\leqslant 0}K_{S}\left( \mathbf{v};%
\mathbf{u}\right) f_{-}(t,\mathbf{x};\mathbf{u})d\mathbf{u}
\label{boundary condition: Lateral}
\end{equation}%
for $\mathbf{v}\cdot \mathbf{n}_{\mathbf{x}}\leq 0$, where $K_{S}\geq 0$ is
the lateral collision kernel and $\alpha _{L}\in \lbrack 0,1].$ We assume 
\begin{equation*}
K_{S}\left( v_{x},v_{\bot };u_{x},u_{\bot }\right) =K_{S}\left( \pm
v_{x},v_{\bot };\pm u_{x},u_{\bot }\right)
\end{equation*}%
since we want the same reflection law for the particles coming from the left
and the right. 
(We require neither the same kernel nor the same $\alpha $ as at the ends.) 
Notice that the horizontal speed $V(t)$ of the body does not enter the
lateral boundary condition. See Subsection \ref{Sec:LateralPartIncludingAoki}
for a different condition.

Because the body moves only horizontally, no particle can collide on the
lateral side more than once. Thus the particles that collide with the
lateral side must have moved in a straight line from $t=0$. So we can put $%
f_{0}$ in place of $f_{-}$ in (\ref{boundary condition: Lateral}), that is, 
\begin{equation}
f_{+}(t,\mathbf{x};\mathbf{v})=\alpha _{S}f_{0}(v_{x},\mathbf{v\cdot T}_{%
\mathbf{x}},-\mathbf{v\cdot n}_{\mathbf{x}})+(1-\alpha _{S})\int_{\mathbf{u}%
\cdot \mathbf{n}_{\mathbf{x}}\leqslant 0}K_{S}\left( \mathbf{v};\mathbf{u}%
\right) f_{0}(\mathbf{u})d\mathbf{u.}  \label{lateral f_0}
\end{equation}

\begin{lemma}
\label{lateral mass} If the lateral collision kernel $K_{S}$ satisfies 
\begin{equation}
\int_{\mathbf{v\cdot n}_{\mathbf{x}}\geqslant 0}\left( \mathbf{v\cdot n}_{%
\mathbf{x}}\right) K_{S}\left( \mathbf{v};\mathbf{u}\right) d\mathbf{v}%
=\left\vert \mathbf{u\cdot n}_{\mathbf{x}}\right\vert \chi _{\mathbf{u\cdot n%
}_{\mathbf{x}}\leqslant 0},  \label{kernelmass:lateral}
\end{equation}%
then the mass is conserved across the lateral boundary. That is, 
\begin{equation*}
\int_{\partial \Omega _{S}(t)}dS_{\mathbf{x}}\int_{{\mathbb{R}}^{3}}d\mathbf{%
v}\left( \mathbf{v}\cdot \mathbf{n}_{x}\right) f(t,\mathbf{x},\mathbf{v})=0.
\end{equation*}
\end{lemma}

\begin{proof}
This is a direct computation. We write%
\begin{eqnarray*}
&&\int_{\partial \Omega _{S}(t)}dS_{\mathbf{x}}\int_{{\mathbb{R}}^{3}}d%
\mathbf{v}\left( \mathbf{v}\cdot \mathbf{n}_{x}\right) f(t,\mathbf{x},%
\mathbf{v}) \\
&=&\int_{\partial \Omega _{S}(t)}dS_{\mathbf{x}}\int_{\mathbf{v}\cdot 
\mathbf{n}_{x}\leqslant 0}d\mathbf{v}\left( \mathbf{v}\cdot \mathbf{n}%
_{x}\right) f_{-}(t,\mathbf{x},\mathbf{v})+\int_{\partial \Omega _{S}(t)}dS_{%
\mathbf{x}}\int_{\mathbf{v}\cdot \mathbf{n}_{x}\geqslant 0}d\mathbf{v}\left( 
\mathbf{v}\cdot \mathbf{n}_{x}\right) f_{+}(t,\mathbf{x},\mathbf{v}) \\
&=&I+II.
\end{eqnarray*}%
Then 
\begin{eqnarray*}
II &=&\alpha _{S}\int_{\partial \Omega _{S}(t)}dS_{\mathbf{x}}\int_{\mathbf{v%
}\cdot \mathbf{n}_{x}\geqslant 0}d\mathbf{v}\left( \mathbf{v}\cdot \mathbf{n}%
_{x}\right) f_{0}(v_{x},\mathbf{v\cdot T}_{\mathbf{x}},-\mathbf{v\cdot n}_{%
\mathbf{x}}) \\
&&+(1-\alpha _{S})\int_{\partial \Omega _{S}(t)}dS_{\mathbf{x}}\int_{\mathbf{%
v}\cdot \mathbf{n}_{x}\geqslant 0}d\mathbf{v}\left( \mathbf{v}\cdot \mathbf{n%
}_{x}\right) \int_{\mathbf{u}\cdot \mathbf{n}_{\mathbf{x}}\leqslant 0}d%
\mathbf{u}\ K_{S}\left( \mathbf{v};\mathbf{u}\right) f_{0}(\mathbf{u}) \\
&=&-\alpha _{S}\int_{\partial \Omega _{S}(t)}dS_{\mathbf{x}}\int_{\mathbf{v}%
\cdot \mathbf{n}_{x}\leqslant 0}d\mathbf{v}\left( \mathbf{v}\cdot \mathbf{n}%
_{x}\right) f_{0}(\mathbf{v}) \\
&&+(1-\alpha _{S})\int_{\partial \Omega _{S}(t)}dS_{\mathbf{x}}\int_{\mathbf{%
u}\cdot \mathbf{n}_{\mathbf{x}}\leqslant 0}d\mathbf{u}f_{0}(\mathbf{u}%
)\left( \int_{\mathbf{v}\cdot \mathbf{n}_{x}\geqslant 0}d\mathbf{v}\left( 
\mathbf{v}\cdot \mathbf{n}_{x}\right) K_{S}\left( \mathbf{v};\mathbf{u}%
\right) \right) \\
&=&-\alpha _{S}\int_{\partial \Omega _{S}(t)}dS_{\mathbf{x}}\int_{\mathbf{v}%
\cdot \mathbf{n}_{x}\leqslant 0}d\mathbf{v}\left( \mathbf{v}\cdot \mathbf{n}%
_{x}\right) f_{0}(\mathbf{v}) \\
&&-(1-\alpha _{S})\int_{\partial \Omega _{S}(t)}dS_{\mathbf{x}}\int_{\mathbf{%
u}\cdot \mathbf{n}_{\mathbf{x}}\leqslant 0}d\mathbf{u}\left( \mathbf{u\cdot n%
}_{\mathbf{x}}\right) f_{0}(\mathbf{u})=-I
\end{eqnarray*}%
by \eqref{lateral f_0}.
\end{proof}

\begin{lemma}
\label{lateral force} If the lateral collision kernel $K_{S}$ satisfies (\ref%
{kernelmass:lateral}), then the contribution $F_{S}(t)$ to the horizontal
force from the lateral boundary vanishes.
\end{lemma}

\begin{proof}
Recalling (\ref{force:lateral}), we have 
\begin{equation*}
F_{S}(t)=\int_{\partial \Omega _{S}(t)}dS_{\mathbf{x}}\int_{\mathbf{v}\cdot 
\mathbf{n\leqslant }0}d\mathbf{v}\left( \mathbf{v}\cdot \mathbf{n}\right)
\left( v_{x}-V\right) f(t,\mathbf{x},\mathbf{v}).
\end{equation*}%
Under \eqref{lateral f_0}, $f_{-}=f_{0}$ and $f_{+}$ are both independent of 
$V,$ so that 
\begin{equation*}
\frac{\partial F_{S}(t)}{\partial V}=-\int_{\partial \Omega _{S}(t)}dS_{%
\mathbf{x}}\int_{\mathbb{R}^{3}}d\mathbf{v}\left( \mathbf{v}\cdot \mathbf{n}%
\right) f(t,\mathbf{x},\mathbf{v})=0
\end{equation*}%
by mass conservation. That is, we can put $V=0$ when we compute $F_{S}(t).$
So 
\begin{eqnarray*}
F_{S}(t) &=&\int_{\partial \Omega _{S}(t)}dS_{\mathbf{x}}\int_{\mathbf{v}%
\cdot \mathbf{n\leqslant }0}d\mathbf{v}\left( \mathbf{v}\cdot \mathbf{n}%
\right) v_{x}f_{0}(\mathbf{v}) \\
&&+\alpha _{S}\int_{\partial \Omega _{S}(t)}dS_{\mathbf{x}}\int_{\mathbf{v}%
\cdot \mathbf{n\geqslant }0}d\mathbf{v}\left( \mathbf{v}\cdot \mathbf{n}%
\right) v_{x}f_{0}(v_{x},\mathbf{v\cdot T}_{\mathbf{x}},-\mathbf{v\cdot n}_{%
\mathbf{x}}) \\
&&+(1-\alpha _{S})\int_{\partial \Omega _{S}(t)}dS_{\mathbf{x}}\int_{\mathbf{%
v}\cdot \mathbf{n}_{x}\geqslant 0}d\mathbf{v}\left( \mathbf{v}\cdot \mathbf{n%
}_{x}\right) v_{x}\int_{\mathbf{u}\cdot \mathbf{n}_{\mathbf{x}}\leqslant 0}d%
\mathbf{u}K_{S}\left( \mathbf{v};\mathbf{u}\right) f_{0}(\mathbf{u}) \\
&=&I+II+III.
\end{eqnarray*}%
By the assumption that $f_{0}$ is even in $v_{x}$, we have 
\begin{equation*}
\int_{\mathbb{R}}v_{x}f_{0}(\mathbf{v})dv_{x}=0.
\end{equation*}%
Thus both $I$ and $II$ are $0$. Notice that $K_{S}\left( v_{x},v_{\bot
};u_{x},u_{\bot }\right) =K_{S}\left( -v_{x},v_{\bot };u_{x},u_{\bot
}\right) ,$ we have 
\begin{equation*}
\int_{\mathbb{R}}v_{x}K_{S}\left( \mathbf{v};\mathbf{u}\right) dv_{x}=0
\end{equation*}%
so that $III=0.$
\end{proof}


\subsubsection{Alternative Boundary Conditions on the Lateral Boundary \label%
{Sec:LateralPartIncludingAoki}}

Assume on $\partial \Omega _{S}$ that%
\begin{equation}
f_{+}(t,\mathbf{x};\mathbf{v})=\int_{\mathbf{u}\cdot \mathbf{n}_{\mathbf{x}%
}\leqslant 0}K_{S}\left( \mathbf{v-i}V(t);\mathbf{u-i}V(t)\right) f_{-}(t,%
\mathbf{x};\mathbf{u})d\mathbf{u},  \label{boundary condition: Lateral2}
\end{equation}%
where%
\begin{equation*}
K_{S}\left( v_{x},v_{\bot };u_{x},u_{\bot }\right) =K_{S}\left( \pm
v_{x},v_{\bot };\pm u_{x},u_{\bot }\right) \geq 0.
\end{equation*}%
For convenience, we have dropped the specular part since $V(t)$ is unrelated
to the specular reflections on $\partial \Omega _{S}.$ A special case of
boundary condition \eqref{boundary condition: Lateral2} was studied in \cite%
{Ital3}. We still have conservation of mass if we assume (\ref%
{kernelmass:lateral}) for $K_S$. However, for the alternative boundary
condition \eqref{boundary
condition: Lateral2} the force does not vanish, as we now show.

\begin{lemma}
\label{alt lateral force} Under boundary condition (\ref{boundary condition:
Lateral2}), the lateral force 
\begin{equation}
F_{S}(t)=\int_{\partial \Omega _{S}(t)}dS_{\mathbf{x}}\int_{\mathbf{v}\cdot 
\mathbf{n\leqslant }0}d\mathbf{v}\left( \mathbf{v}\cdot \mathbf{n}_{\mathbf{x%
}}\right) \left( v_{x}-V(t)\right) f_{0}(\mathbf{v})
\label{force:constant lateral force}
\end{equation}%
is a nonnegative function depending solely on $V(t)$. It satisfies $\frac{%
\partial F_{S}}{\partial V}\geqslant 0.$
\end{lemma}

\begin{proof}
Again there is no recollision on $\partial \Omega _{S}$, so we have 
\begin{eqnarray*}
F_{S}(t) &=&\int_{\partial \Omega _{S}(t)}dS_{\mathbf{x}}\int_{\mathbf{v}{\ }%
\cdot \mathbf{n\leqslant }0}d\mathbf{v}\left( \mathbf{v}\cdot \mathbf{n}_{%
\mathbf{x}}\right) \left( v_{x}-V\right) f_{0}(\mathbf{v}) \\
&&+\int_{\partial \Omega _{S}(t)}dS_{\mathbf{x}}\int_{\mathbf{v}\cdot 
\mathbf{n\geqslant }0}d\mathbf{v}\left( \mathbf{v}\cdot \mathbf{n}_{\mathbf{x%
}}\right) \left( v_{x}-V\right) \int_{\mathbf{u}\cdot \mathbf{n}_{\mathbf{x}%
}\leqslant 0}K_{S}\left( \mathbf{v-i}V(t);\mathbf{u-i}V(t)\right) f_{0}(%
\mathbf{u})d\mathbf{u} \\
&=&I+II.
\end{eqnarray*}%
A change of variable of $v_{x}$ gives 
\begin{equation*}
II=\int_{\partial \Omega _{S}(t)}dS_{\mathbf{x}}\int_{\mathbf{v}\cdot 
\mathbf{n\geqslant }0}d\mathbf{v}\left( \mathbf{v}\cdot \mathbf{n}_{\mathbf{x%
}}\right) v_{x}\int_{\mathbf{u}\cdot \mathbf{n}_{\mathbf{x}}\leqslant
0}K_{S}\left( \mathbf{v};\mathbf{u-i}V(t)\right) f_{0}(\mathbf{u})d\mathbf{u}%
=0
\end{equation*}%
because $K_{S}(\mathbf{v},\mathbf{u})$ is even in $v_{x}$. Thus $%
F_{S}(t)=G(V(t))$ where 
\begin{equation*}
G(V)=\int_{\partial \Omega _{S}(t)}dS_{\mathbf{x}}\int_{\mathbf{v}\cdot 
\mathbf{n}_{\mathbf{x}}\mathbf{\leqslant }0}d\mathbf{v}\left( \mathbf{v}%
\cdot \mathbf{n}_{\mathbf{x}}\right) \left( v_{x}-V\right) f_{0}(\mathbf{v}).
\end{equation*}%
Now 
\begin{equation*}
G(0)=\int_{\partial \Omega _{S}(t)}dS_{\mathbf{x}}\int_{\mathbf{v}\cdot 
\mathbf{n}_{\mathbf{x}}\mathbf{\leqslant }0}d\mathbf{v}\left( \mathbf{v}%
\cdot \mathbf{n}_{\mathbf{x}}\right) v_{x}f_{0}(\mathbf{v})=0
\end{equation*}%
since $\int_{\mathbb{R}}v_{x}f_{0}(\mathbf{v})d\mathbf{v}=0$ and 
\begin{equation*}
G^{\prime }(V)=-\int_{\partial \Omega _{S}(t)}dS_{\mathbf{x}}\int_{\mathbf{v}%
{\ }\cdot \mathbf{n}_{\mathbf{x}}\mathbf{\leqslant }0}d\mathbf{v}\left( 
\mathbf{v}{\ }\cdot \mathbf{n}_{\mathbf{x}}\right) f_{0}(\mathbf{v}%
)\geqslant 0.
\end{equation*}%
So $G\geqslant 0$, which means $F_{S}(t)\geqslant 0.$
\end{proof}

No matter whether we take (\ref{boundary condition: Lateral}) or (\ref%
{boundary condition: Lateral2}) as the boundary condition on $\partial
\Omega _{S}$, $F_{S}(t)$ is a nonnegative function which depends solely on $%
V(t).$ So we write $F_{S}(t)$ as $F_{S}(V(t))$ from here on.


\subsection{Total Force on the Body}

We now use the boundary conditions to write the force explicitly and
succinctly in terms of $f_{-}(t,x,v)$.

\begin{lemma}
\label{total force} The force is given in terms of $f_{-}(t,\mathbf{x},%
\mathbf{v})$ as 
\begin{equation*}
F(t)=F_{S}(V(t))+\int_{\partial \Omega_R (t) \cup \partial \Omega_L (t)}
dS_{x}\int_{\mathbb{R}^{3}}d\mathbf{v}\ \text{sgn}(V(t)-v_{x})\ \ell (%
\mathbf{v}-iV(t))\ f_{-}(t,\mathbf{x},\mathbf{v}),
\end{equation*}%
where the nonnegative function $F_{S}(V(t))$ either vanishes or is given by (%
\ref{force:constant lateral force}) depending on the choice of boundary
condition on $\partial \Omega _{l}$, and 
\begin{equation}  \label{ell}
\ell (\mathbf{w})=(1+\alpha )w_{x}^{2}+(1-\alpha )\int_{v_{x}\geq 0}d\mathbf{%
v}\ v_{x}^{2}\ K(\mathbf{v},\mathbf{w}).
\end{equation}%
Of course, the last integral could also be taken over $\{v_{x}\leq 0\}$ due
to the evenness of the kernel.
\end{lemma}

\begin{proof}
Recall that $F_{L,R}(t)$ is 
\begin{eqnarray*}
F_{L,R}(t) &=&\int_{\partial \Omega _{R}(t)}dS_{\mathbf{x}%
}\int_{v_{x}\leqslant V_{x}(t)}d\mathbf{v}[v_{x}-V_{x}(t)]^{2}f_{-}\left( t,%
\mathbf{x},\mathbf{v}\right) \\
&&+\int_{\partial \Omega _{R}(t)}dS_{\mathbf{x}}\int_{v_{x}\geqslant
V_{x}(t)}d\mathbf{v}[v_{x}-V_{x}(t)]^{2}f_{+}\left( t,\mathbf{x},\mathbf{v}%
\right) \\
&&-\int_{\partial \Omega _{L}(t)}dS_{\mathbf{x}}\int_{v_{x}\geqslant
V_{x}(t)}d\mathbf{v}[v_{x}-V_{x}(t)]^{2}f_{-}\left( t,\mathbf{x},\mathbf{v}%
\right) \\
&&-\int_{\partial \Omega _{L}(t)}dS_{\mathbf{x}}\int_{v_{x}\leqslant
V_{x}(t)}d\mathbf{v}[v_{x}-V_{x}(t)]^{2}f_{+}\left( t,\mathbf{x},\mathbf{v}%
\right)
\end{eqnarray*}%
Plugging in the boundary conditions (\ref{boundary condition: Right}) and (%
\ref{boundary condition: Left}), it becomes%
\begin{eqnarray}
F_{L,R}(t) &=&\int_{\partial \Omega _{R}\left( t\right) }dS_{\mathbf{x}%
}\int_{v_{x}\leqslant V(t)}d\mathbf{v}\left( V\left( t\right) -v_{x}\right)
^{2}f_{-}(t,\mathbf{x};\mathbf{v})-\int_{\partial \Omega _{L}\left( t\right)
}dS_{\mathbf{x}}\int_{v_{x}\geqslant V(t)}d\mathbf{v}\left( V\left( t\right)
-v_{x}\right) ^{2}f_{-}(t,\mathbf{x};\mathbf{v})  \notag \\
&&+\alpha \int_{\partial \Omega _{R}\left( t\right) }dS_{\mathbf{x}%
}\int_{v_{x}\geqslant V(t)}d\mathbf{v}\left( v_{x}-V\left( t\right) \right)
^{2}f_{-}(t,\mathbf{x};2V(t)-v_{x},v_{\bot })  \notag \\
&&+(1-\alpha )\int_{\partial \Omega _{R}\left( t\right) }dS_{\mathbf{x}%
}\int_{v_{x}\geqslant V(t)}d\mathbf{v}\left( v_{x}-V\left( t\right) \right)
^{2}\int_{u_{x}\leqslant V\left( t\right) }d\mathbf{u}K\left( \mathbf{v}-%
\mathbf{i}V\left( t\right) ;\mathbf{u}-\mathbf{i}V\left( t\right) \right)
f_{-}(t,\mathbf{x};\mathbf{u})  \notag \\
&&-\alpha \int_{\partial \Omega _{L}\left( t\right) }dS_{\mathbf{x}%
}\int_{v_{x}\leqslant V(t)}d\mathbf{v}\left( v_{x}-V\left( t\right) \right)
^{2}f_{-}(t,\mathbf{x};2V(t)-v_{x},v_{\bot })  \notag \\
&&-(1-\alpha )\int_{\partial \Omega _{L}\left( t\right) }dS_{\mathbf{x}%
}\int_{v_{x}\leqslant V(t)}d\mathbf{v}\left( v_{x}-V\left( t\right) \right)
^{2}\int_{u_{x}\geqslant V\left( t\right) }d\mathbf{u}K\left( \mathbf{v}-%
\mathbf{i}V\left( t\right) ;\mathbf{u}-\mathbf{i}V\left( t\right) \right)
f_{-}(t,\mathbf{x};\mathbf{u})  \notag \\
&=&\int_{\partial \Omega _{R}\left( t\right) }dS_{\mathbf{x}%
}\int_{v_{x}\leqslant V(t)}d\mathbf{v}\ \ell \left( \mathbf{v}-\mathbf{i}%
V\left( t\right) \right) f_{-}(t,\mathbf{x};\mathbf{v})  \notag \\
&&-\int_{\partial \Omega _{L}\left( t\right) }dS_{\mathbf{x}%
}\int_{v_{x}\geqslant V(t)}d\mathbf{v}\ \ell \left( \mathbf{v}-\mathbf{i}%
V\left( t\right) \right) f_{-}(t,\mathbf{x};\mathbf{v}),
\label{formula:F(t) for product kernel}
\end{eqnarray}%
where $\ell (\mathbf{w})$ is defined in \eqref{ell}.
\end{proof}

\subsubsection{Force without Recollisions}

Putting the initial density $f_{0}(\mathbf{v})$ at the place of $f_{-}(t,%
\mathbf{x};\mathbf{v})$ in formula (\ref{formula:F(t) for product kernel}),
we get the \textit{fictitious force} 
\begin{equation*}
F_{0}(V) = F_{S}(V)+C\left( \int_{v_{x}\leqslant V}\ell(\mathbf{v}-\mathbf{i}%
V)f_{0}(\mathbf{v})d\mathbf{v}-\int_{v_{x}\geqslant V}\ell(\mathbf{v}-%
\mathbf{i}V)f_{0}(\mathbf{v})d\mathbf{v}\right) ,
\end{equation*}
where $C$ is the area of the ends of the cylinder. This is the force on the
cylinder if all the collisions occurring before time $t$ were ignored. The
basic properties of the fictitious force are stated in the next lemma.

\begin{lemma}
\label{fictitious lemma} Suppose $f_{0}(\mathbf{v})\ge0$\ is even,
continuous and $\not\equiv0$. If $\ell\in C^1$ and $\partial _{w_{x}}\ell(%
\mathbf{w}) < 0$ 
for $w_{x}\in (-\infty ,0)$, then $F_{0}(V)$ is a positive, increasing $%
C^{1} $ function of $V$.
\end{lemma}

\begin{proof}
First we have 
\begin{eqnarray*}
F_{0}(V) &=&F_{S}(V)+C\left( \int_{v_{x}\leqslant V}\ell(\mathbf{v}-\mathbf{i%
}V)f_{0}(\mathbf{v})d\mathbf{v}-\int_{v_{x}\geqslant V}\ell(\mathbf{v}-%
\mathbf{i}V)f_{0}(\mathbf{v})d\mathbf{v}\right) \\
&=&F_{S}(V)+C\left( \int_{v_{x}\leqslant V}\ell(\mathbf{v}-\mathbf{i}V)f_{0}(%
\mathbf{v})d\mathbf{v}-\int_{v_{x}\leqslant -V}\ell(\mathbf{v}+\mathbf{i}%
V)f_{0}(\mathbf{v})d\mathbf{v}\right) \\
&\geqslant &C\int_{v_{x}\leqslant -V}\left( \ell(\mathbf{v}-\mathbf{i}%
V)-\ell( \mathbf{v}+\mathbf{i}V)\right) f_{0}(\mathbf{v})d\mathbf{v} > 0,
\end{eqnarray*}
because $v_{x}-V\leqslant v_{x}+V$ and $\partial _{w_{x}}\ell(\mathbf{w}) <
0 $ for $w_{x}\in (-\infty ,0).$ Using the monotonicity of $\ell$ again, we
deduce the monotonicity of $F_{0}(V)$ by%
\begin{eqnarray*}
F_{0}^{\prime }(V) &=&-\int_{\mathbf{v}\cdot \mathbf{n\leqslant }0}d\mathbf{v%
}\left( \mathbf{v}\cdot \mathbf{n}\right) f_{0}(\mathbf{v}) \\
&& +C\left( -\int_{v_{x}\leqslant V}\left( \partial _{w_{x}}\ell\right) (%
\mathbf{v}-\mathbf{i}V)f_{0}(\mathbf{v})d\mathbf{v}-\int_{v_{x}\leqslant
-V}\left( \partial _{w_{x}}\ell\right) (\mathbf{v}+\mathbf{i}V)f_{0}(\mathbf{%
v})d\mathbf{v}\right) > 0.
\end{eqnarray*}
\end{proof}


\section{Iteration Scheme}

\subsection{The Iteration Family}

\begin{definition}
\label{def:W}We define $\mathcal{W}$ as the family of functions $W$ that
satisfy the following conditions.

(i) $W:[0,\infty )\rightarrow \mathbb{R}$ is Lipschitz and $W(0)=V_{0}$.

(ii) $W$ is strictly increasing over the interval $[0,t_{0}]$ for some $%
t_{0} $ depending on $\gamma =V_{\infty }-V_{0}.$

(iii) There exist bounded functions $h(t)=h(t,\gamma)$ and $g(t)=g(t,\gamma)$
such that for all $W\in \mathcal{W}$, $t\in \lbrack 0,\infty )$ and $\gamma
\in (0,1)$, 
\begin{equation}
0<\gamma h(t,\gamma )\leqslant V_{\infty }-W(t) < \gamma g(t,\gamma ).
\label{g and h}
\end{equation}
\end{definition}

We do not assume that $W(t)$ is increasing in $[t_0,\infty)$. Specific
choices for the functions $g$ and $h$ will be made later.  For any function $%
Y:[0,\infty )\rightarrow \mathbb{R}$, we denote its average over time
intervals by 
\begin{equation*}
\left\langle Y\right\rangle _{s,t}=\frac{1}{t-s}\int_{s}^{t}Y\left( \tau
\right) d\tau ,\qquad \left\langle Y\right\rangle _{0,t}=\left\langle
Y\right\rangle _{t}.
\end{equation*}%
Thus $Y\in L^1(\mathbb{R})$ implies $\langle Y \rangle_t = O(1/t)$. The
family $\mathcal{W}=\left\{ W\right\} $ has the following properties.

\begin{lemma}
\label{Lemma:The class of W} If $\langle h\rangle _{t}>g(t)$ for all $%
t\geqslant t_{0}$, then

(i) \ $W\left( t\right) >\left\langle W\right\rangle _{t}$, $\forall t>0.$

(ii) \ $\left\langle W\right\rangle _{t}$ is an increasing function. In
particular, $V_{0}\leqslant \left\langle W\right\rangle _{t}\leqslant
W\left( t\right) \leqslant V_{\infty }.$

(iii) \ $\left\langle W\right\rangle _{s,t}>\left\langle W\right\rangle _{t}$%
, $\forall s\in \left( 0,t\right) .$

(iv) \ $\gamma \left[ \langle h\rangle _{t}-g\left( t\right) \right] \ \chi
\{t\geqslant t_{0}\}\leqslant W(t)-\left\langle W\right\rangle _{t}\leqslant
\gamma \langle g\rangle _{t}. $
\end{lemma}

\begin{proof}
When $t\leqslant t_{0}$, (i) follows from the assumption that $W$ is
increasing. When $t\geqslant t_{0}$, we have%
\begin{eqnarray*}
W(t)-\left\langle W\right\rangle _{t} &=&\frac{1}{t}\int_{0}^{t}\left[
\left( V_{\infty }-W\left( \tau \right) \right) -\left( V_{\infty }-W\left(
t\right) \right) \right] d\tau \\
&\geqslant &\frac{1}{t}\int_{0}^{t}\left[ \gamma h(\tau)-\gamma g(t)\right]
d\tau \\
&\geqslant &\gamma \left( \frac{1}{t}\int_{0}^{t}h(\tau)d\tau-g(t)\right) >0
\end{eqnarray*}%
by assumption. This proves (i) and part of (iv). Now

\begin{equation*}
\frac{d}{dt}\left\langle W\right\rangle _{t}=\frac{1}{t}\left( -\left\langle
W\right\rangle _{t}+W\left( t\right) \right) >0
\end{equation*}%
by (i). Thus (ii) is true. Moreover,

\begin{eqnarray*}
\left\langle W\right\rangle _{s,t}-\left\langle W\right\rangle _{t} &=&\frac{%
1}{t-s}\int_{s}^{t}W\left( \tau \right) d\tau -\frac{1}{t}%
\int_{0}^{t}W\left( \tau \right) d\tau \\
&=&\frac{1}{t-s}\int_{0}^{t}W\left( \tau \right) d\tau -\frac{1}{t-s}%
\int_{0}^{s}W\left( \tau \right) d\tau -\frac{1}{t}\int_{0}^{t}W\left( \tau
\right) d\tau \\
&=&\frac{s}{t-s}\left( \frac{1}{t}\int_{0}^{t}W\left( \tau \right) d\tau -%
\frac{1}{s}\int_{0}^{s}W\left( \tau \right) d\tau \right) >0
\end{eqnarray*}%
by (ii). Finally, 
\begin{eqnarray*}
W(t)-\left\langle W\right\rangle _{t} &=&\frac{1}{t}\int_{0}^{t}\left[
\left( V_{\infty }-W\left( \tau \right) \right) -\left( V_{\infty }-W\left(
t\right) \right) \right] d\tau \\
&\leqslant &\frac{1}{t}\int_{0}^{t}\left[ \gamma g(\tau )-\gamma h(t)\right]
d\tau \leqslant \frac{\gamma }{t}\int_{0}^{t}g(\tau )d\tau .
\end{eqnarray*}
\end{proof}

The key step in the proof of the theorems will be to prove that $%
r_{W}^{R}\left( t\right) +r_{W}^{L}\left( t\right) \geqslant 0$ where we
define 
\begin{eqnarray*}
r_{W}^{R}\left( t\right) &=&\int_{\partial \Omega _{R}\left( t\right) }dS_{%
\mathbf{x}}\int_{u_{x}\leqslant W\left( t\right) }d\mathbf{u}\ \ell(\mathbf{u%
}-\mathbf{i}W(t))\left\{ f_{-}(t,\mathbf{x},\mathbf{u})-f_{0}(\mathbf{u}%
)\right\} , \\
r_{W}^{L}\left( t\right) &=&\int_{\partial \Omega _{L}\left( t\right) }dS_{%
\mathbf{x}}\int_{u_{x}\geqslant W\left( t\right) }d\mathbf{u}\ \ell(\mathbf{u%
}-\mathbf{i}W(t))\left\{ f_{-}(t,\mathbf{x},\mathbf{u})-f_{0}(\mathbf{u}%
)\right\} .
\end{eqnarray*}%
They represent the forces on the right and left of the cylinder due to the 
\textit{precollisions}, that is, all the collisions occurring \textit{before}
time $t$. This will be accomplished via a lower bound of $r_{W}^{R}\left(
t\right) $ (Lemma \ref{Lemma:Upper and lower bound of R+}) and an upper
bound of $\left\vert r_{W}^{L}\left( t\right) \right\vert $ (Lemma \ref%
{Lemma:UpperBoundOfR-}). Then we will be able to determine $g$ and $h$ via
the requirement that $\mathcal{W}=\left\{ W\right\} $ is closed under the
map $W \rightarrow V_{W}$.

Before beginning the detailed estimates, we consider the meaning of a
precollision. In order for a particle to have collisions at two times $t$
and $s$ with $s<t$, it is obviously required that 
\begin{equation*}
\int_s^t v(\tau) d\tau = \int_s^t W(\tau) d\tau .
\end{equation*}
In order to have no collisions in between $s$ and $t$, it is necessary that 
\begin{equation*}
(t-s)v_{x}=\int_{s}^{t}W(\tau )d\tau ,\qquad\left\vert v_{\bot }\right\vert
\leqslant \frac{2r}{t-s} ,
\end{equation*}%
where $r$ is the radius of the cylinder. Since $\left\langle W\right\rangle
_{s,t}$ is a continuous function of $s$ for any $t$, the existence of a
precollision at some time earlier than $t$ requires that 
\begin{eqnarray}  \label{velocitybound}
v_{x} &\in &\left[ \inf\limits_{s<t}\left\langle W\right\rangle
_{s,t},\sup\limits_{s<t}\left\langle W\right\rangle _{s,t}\right] = \left[
\left\langle W\right\rangle _{t},\sup\limits_{s<t}\left\langle
W\right\rangle _{s,t}\right] ,  \label{condition:recollision condition} \\
\left\vert v_{\bot }\right\vert &\leqslant &\frac{2r}{t-s} .  \notag
\end{eqnarray}
We will estimate $r_{W}^{R}\left( t\right) $ and $r_{W}^{L}\left( t\right) $
by taking only one precollision into account.

\subsection{Assumptions on $K$ and $f_{0}$}

We make the following assumptions on the collision kernel $K$ for the ends
of the cylinder and on the initial particle density $f_{0}$, in addition to
the previously stated assumptions that $K(\mathbf{v},\mathbf{u})$ and $f_0(%
\mathbf{v})$ are nonnegative, even in $v_x$ and $u_x$, and \eqref{kernelmass}
is valid. The first assumption below implies that at the boundary the
momentum is transferred only horizontally.

\noindent \textbf{A1.} Let $K$ and $f_{0}$ have the product form 
\begin{eqnarray*}
f_{0}(\mathbf{v}) &=&a_{0}(v_{x})b(v_{\bot }), \\
K(\mathbf{v,u}) &=&k(v_{x},u_{x})b(v_{\bot }),\qquad \int b(v_{\bot
})dv_{\bot }=1,
\end{eqnarray*}%
with each factor nonnegative and continuous and $f_{0}$ bounded.

Thus $a_0$ and $k$ are even. Under Assumption A1, $\ell(\mathbf{w})$
actually depends only on $w_{x}$, that is, 
\begin{equation*}
\ell (\mathbf{w})=(1+\alpha )w_{x}^{2}+(1-\alpha )\int_{v_{x}\geq 0}dv_x \
v_{x}^{2}\ k(v_x, w_x) = \ell(w_x).
\end{equation*}
Therefore, at any later time, $f_{+}$ and $f_{-}$ must take the product form 
\begin{equation*}
f_{+}(t,\mathbf{x};\mathbf{v}) = a_{+}(t,\mathbf{x;}v_{x})b(v_{\bot }),
\quad f_{-}(t,\mathbf{x};\mathbf{v}) = a_{-}(t,\mathbf{x;}v_{x})b(v_{\bot }).
\end{equation*}
It is then natural to ask whether the analysis is purely one-dimensional. In
fact, the dimension does come into play as will be demonstrated in Lemmas %
\ref{Lemma:Upper and lower bound of R+} and \ref{Lemma:UpperBoundOfR-}.
\bigskip

\noindent \textbf{A2.}%
\begin{equation*}
\sup_{\left\vert u_{x}\right\vert \leqslant \gamma }\sup_{v_{x}\in \mathbb{R}%
}k(v_{x},u_{x})<\infty .
\end{equation*}

\noindent \textbf{A3.} There is a power $0\leq p\leq 2$ and there are
positive constants $C$ and $c$ such that 
\begin{equation*}
c\left\vert u_{x}\right\vert ^{p}\leqslant \ \int_{v_{x}\geq 0}\ v_{x}^{2}\
k(v_{x},u_{x})\ dv_{x}\ \leqslant C\left\vert u_{x}\right\vert ^{p}
\end{equation*}%
for $u_{x}\in \left[ -\gamma ,0\right) $ We also assume that this integral
is a $C^{1}$ function of $u_{x}$ for $u_{x}\neq 0$. Note that A3 and A1
imply that $(1-\alpha )c|u_{x}|^{p}\leq \ell (u_{x})\leq (1+\alpha
)u_{x}^{2}+(1-\alpha )C|u_{x}|^{p}$.

\bigskip

\noindent \textbf{A4.}%
\begin{equation*}
\sup_{v_{x}\in \lbrack -\gamma ,0]}\ \sup_{\eta \in \lbrack V_{0},V_{\infty
}]}\int_{-\infty }^{V_{\infty }}k(v_{x},u_{x}-\eta
)a_{0}(u_{x})du_{x}<\infty .
\end{equation*}

\noindent \textbf{A5.}There exists $\delta >0$ such that 
\begin{equation*}
(1-\alpha )\inf_{v_{x}\in \lbrack -\gamma ,0]}\inf_{\eta \in \lbrack
V_{0},V_{\infty }]}\int_{-\infty }^{V_{0}}k(v_{x},u_{x}-\eta
)a_{0}(u_{x})du_{x}\geq a_{0}(V_{\infty })+\delta .
\end{equation*}

\section{Examples of Collision Kernels}

\begin{example}
\label{example:aoki}The kernel 
\begin{equation*}
K(\mathbf{v},\mathbf{u})=C_{2}e^{-\beta \left\vert \mathbf{v}\right\vert
^{2}}\left\vert u_{x}\right\vert
\end{equation*}%
and the initial density 
\begin{equation*}
f_{0}(\mathbf{v})=C_{1}e^{-\beta \left\vert \mathbf{v}\right\vert ^{2}}
\end{equation*}%
were the subject of \cite{Ital3}. They satisfy all the Assumptions A1-A5.
The constant $C_{2}$ is determined so as to satisfy \eqref{kernelmass}.
Indeed, it is obvious that they satisfy A1-A4. 
In order to verify A5, we note that 
\begin{equation}
(1-\alpha )C_{2}\inf_{v_{x}\in \lbrack -\gamma ,0]}\inf_{\eta \in \lbrack
V_{0},V_{\infty }]}\int_{-\infty }^{V_{0}}e^{-\beta v_{x}^{2}}\left\vert
u_{x}-\eta \right\vert e^{-\beta u_{x}^{2}}du_{x}=(1-\alpha )C_{2}C^{\ast
}e^{-\beta \gamma ^{2}},  \label{verifyA5}
\end{equation}%
where 
\begin{equation*}
C^{\ast }=\inf_{\eta \in \lbrack V_{0},V_{\infty }]}\int_{\infty
}^{V_{0}}(\eta -u_{x})e^{-\beta u_{x}^{2}}du_{x}=V_{0}\int_{\infty
}^{V_{0}}e^{-\beta u_{x}^{2}}du_{x}+\frac{1}{2\beta }e^{-\beta V_{0}^{2}}
\end{equation*}%
depends only on $V_{0}$ and $\beta $. We may consider $V_{\infty }$ as fixed
and $\gamma $ as small and then $V_{0}=V_{\infty }-\gamma $. From (\ref%
{verifyA5}) we require $(1-\alpha )C^{\ast }e^{-\beta \gamma ^{2}}>e^{-\beta
V_{\infty }^{2}}+\delta $ for some $\delta >0$ and all sufficiently small $%
\gamma $. Thus all we require is that 
\begin{equation*}
(1-\alpha )C_{2}\left[ V_{\infty }\int_{\infty }^{V_{\infty }}e^{-\beta
u_{x}^{2}}du_{x}+\frac{1}{2\beta }e^{-\beta V_{\infty }^{2}}\right]
>e^{-\beta V_{\infty }^{2}}.
\end{equation*}%
Because of the second term, A5 is true provided 
\begin{equation}  \label{beta small}
\beta <\frac{1-\alpha }{2}C_{2},
\end{equation}
which is a different kind of condition than in \cite{Ital3}. Using instead
the first term, we note that 
\begin{equation*}
V_{\infty }\int_{\infty }^{V_{\infty }}e^{-\beta u_{x}^{2}}du_{x}\geq
V_{\infty }\int_{-\infty }^{0}e^{-\beta u_{x}^{2}}dx=V_{\infty }\sqrt{\frac{%
\pi }{4\beta }},
\end{equation*}%
so that A5 is also satisfied if 
\begin{equation}  \label{Vinfinity large}
V_{\infty }e^{\beta V_{\infty }^{2}}>\frac{\sqrt{4\beta }}{(1-\alpha )C_{2}%
\sqrt{\pi }}.
\end{equation}
In \cite{Ital3} the condition was that $V_\infty$ be sufficiently large
without specifying how large. The inequality \eqref{Vinfinity large} is a
precise condition.
\end{example}


\begin{example}
\label{example:p=3/2}Choose 
\begin{equation*}
K(\mathbf{v,u})=C_{2}e^{-\frac{v_{x}^{2}}{\left\vert u_{x}\right\vert }%
}b(v_{\bot }),\text{ }f_{0}(\mathbf{v})=a_{0}(v_{x})b(v_{\bot }),
\end{equation*}%
where once again $C_{2}$ is chosen so that \eqref{kernelmass} is satisfied, $%
a_{0}\in L^{1}(\mathbb{R})$, and $\int bdv_{\bot }=1.$ A2 is easily
satisfied, because $0\leqslant e^{-\frac{v_{x}^{2}}{\left\vert
u_{x}\right\vert }}\leqslant 1$. Since 
\begin{equation*}
C_{2}\int_{0}^{\infty }v_{x}^{2}e^{-\frac{v_{x}^{2}}{\left\vert
u_{x}\right\vert }}dv_{x}=C\left\vert u_{x}\right\vert ^{\frac{3}{2}},
\end{equation*}%
A3 is satisfied with $p=\frac{3}{2}.$ We also have%
\begin{equation*}
\sup_{v_{x}}\sup_{\eta }\int_{-\infty }^{V_{\infty }}e^{-\frac{v_{x}^{2}}{%
\left\vert u_{x}-\eta \right\vert }}a_{0}(u_{x})du_{x}\leqslant \int_{%
\mathbb{R}}a_{0}(u_{x})du_{x},
\end{equation*}%
which verifies A4. To test A5, we notice that%
\begin{eqnarray*}
&&C_{2}(1-\alpha )\inf_{v_{x}\in \lbrack -\gamma ,0]}\inf_{\eta \in \lbrack
V_{0},V_{\infty }]}\int_{-\infty }^{V_{0}}e^{-\frac{v_{x}^{2}}{\left\vert
u_{x}-\eta \right\vert }}a_{0}(u_{x})du_{x} \\
&\geqslant &C_{2}(1-\alpha )\inf_{v_{x}\in \lbrack -\gamma ,0]}\inf_{\eta
\in \lbrack V_{0},V_{\infty }]}\int_{-\infty }^{V_{\infty }-1}e^{-\frac{%
v_{x}^{2}}{\left\vert u_{x}-\eta \right\vert }}a_{0}(u_{x})du_{x} \\
&\geqslant &C_{2}(1-\alpha )e^{-\gamma ^{2}}\int_{-\infty }^{V_{\infty
}-1}a_{0}(u_{x})du_{x}.
\end{eqnarray*}%
Thus if, for small enough $\gamma $, we have%
\begin{equation*}
\int_{-\infty }^{V_{\infty }-1}a_{0}(u_{x})du_{x}>\frac{a_{0}(V_{\infty })}{%
C_{2}(1-\alpha )},
\end{equation*}%
then A5 is satisfied.

The physical interpretation of such a choice of kernel is the following.
Notice that 
\begin{equation*}
k\left( v_{x}-V(t),u_{x}-V(t)\right) =C_{2}e^{-\frac{\left(
v_{x}-V(t)\right) ^{2}}{\left\vert u_{x}-V(t)\right\vert }}.
\end{equation*}%
Thus if $\left\vert u_{x}-V(t)\right\vert $ is big, then there is a wide
range of possible emitted velocities. On the other hand, if $\left\vert
u_{x}-V(t)\right\vert $ is small, meaning that the incident particle and the
body move at almost the same speed, then the same is true for the emitted
particles with high probability.
\end{example}

It is then natural to wonder if we can have a family of kernels such that it
covers a continuous range of $p.$ This is simply achieved by modifying
Example \ref{example:p=3/2}.

\begin{example}
\label{example:range of p}For $\beta \in \left[ -1,3\right] $, consider 
\begin{equation*}
K(\mathbf{v,u})=C_{2}\left\vert u_{x}\right\vert ^{\beta } e^{- {v_{x}^{2}} {%
\left\vert u_{x}\right\vert ^{\beta -1}}} b(v_{\bot }), \quad f_{0}(\mathbf{v%
})=a_{0}(v_{x})b(v_{\bot }),
\end{equation*}%
where $C_{2}$ is chosen so that \eqref{kernelmass} is satisfied, while $%
a_{0} $ and $b$ are as in Example \ref{example:p=3/2}. We then have%
\begin{equation*}
C_{2}\left\vert u_{x}\right\vert ^{\beta }\int_{0}^{\infty }v_{x}^{2} \ e^{- 
{v_{x}^{2}} {\left\vert u_{x}\right\vert ^{\beta -1}}} dv_{x}=C\left\vert
u_{x}\right\vert ^{\frac{3-\beta }{2}}.
\end{equation*}
Thus $p$ runs through $\left[ 0,{2}\right] $ as $\beta$ runs through $\left[%
-1,3\right] .$ In particular, if $\beta=1$ we have Example \ref{example:aoki}%
. If $\beta = 0$ we have Example \ref{example:p=3/2}. The same physical
interpretation as in Example \ref{example:p=3/2} holds if $\beta \in \left[
-1,1\right)$, meaning that $p\in(1,2]$.
\end{example}

It is also natural to inquire whether a gaussian is needed. Actually, it
just suffices to have some good decay, as we now illustrate.

\begin{example}
Let us choose 
\begin{equation*}
K(\mathbf{v},\mathbf{u})=C_{2}|u_{x}|\langle v_{x}\rangle ^{-N}\langle
v_{\perp }\rangle ^{-M},\quad f_{0}(\mathbf{v})=\langle v_{x}\rangle
^{-P}\langle v_{\perp }\rangle ^{-M},
\end{equation*}%
where $C_{2}$ is chosen so that \eqref{kernelmass} is satisfied and $%
M>2,N>3,P>2$. Assumption A1 is true because $M>2$. A2 is obvious. A3 is true
because $N>3$. A4 is true because $P>2$. A5 requires 
\begin{equation*}
C_{2}\int_{-\infty }^{V_{0}}(V_{0}-u_{x})\langle u_{x}\rangle
^{-P}du_{x}>\langle V_{\infty }\rangle ^{-P},
\end{equation*}%
which is true for instance if $V_{\infty }$ is sufficiently large. One can
also modify this example to cover a range of $p$ instead of only $p=1$.
\end{example}


\section{Main Estimates of the Force}

\subsection{The Right Side}

In the next lemma we estimate the force on the right side of the cylinder.

\begin{lemma}
\label{Lemma:Upper and lower bound of R+} Let $K$ and $a_{0}$ satisfy the
Assumptions A1-A5. Then for all sufficiently small $\gamma $ we have the
inequalities 
\begin{equation*}
\frac{C\gamma ^{p+1}\chi \{t\geqslant t_{0}\}}{t^{d-1}}\left( \langle
h\rangle _{t}-g\left( t\right) \right) ^{p+1}\leqslant r_{W}^{R}\left(
t\right) \leqslant \frac{C\gamma ^{p+1}\langle g\rangle _{t}^{p+1}}{\left(
1+t\right) ^{d-1}}+C\gamma ^{p+1}\sup_{\frac{t}{2}\leqslant \tau \leqslant
t}\langle g\rangle _{\tau ,t}^{p+1}.
\end{equation*}
\end{lemma}

We remark that it would seem that the second term in the upper bound of $%
r_{W}^{R}\left( t\right) $ should dominate. Such a statement is actually not
true. The first term always acts like $t^{-\left( d+p\right) }$ while the
second one is like $g^{p+1}$. But $g(t)$ has to act like $t^{-\left(
d+p\right) }$ for the sake of the fixed point argument. This will be
clarified in the proof of Corollary \ref{Prop:DeducingConditionsOnhAndg}.

\begin{proof}
To establish upper and lower bounds of $r_{W}^{R}\left( t\right) ,$ we need
upper and lower bounds of $f_{+}(t,\mathbf{x};\mathbf{v}).$ Recall the
boundary condition (\ref{boundary condition: Right}) on the right of the
cylinder%
\begin{equation*}
f_{+}(t,\mathbf{x};\mathbf{v})=\alpha f_{-}(t,\mathbf{x};2W(t)-v_{x},v_{\bot
})+(1-\alpha )\int_{u_{x}\leqslant W\left( t\right) }K\left( \mathbf{v}-%
\mathbf{i}W\left( t\right) ;\mathbf{u}-\mathbf{i}W\left( t\right) \right)
f_{-}(t,\mathbf{x};\mathbf{u})d\mathbf{u}.
\end{equation*}%
In light of condition (\ref{condition:recollision condition}), we denote the
precollision characteristic function by 
\begin{eqnarray*}
\chi _{0}(t,\mathbf{u}) &=&\chi \left\{ \mathbf{u:}\ \forall s\in (0,t),%
\text{ either }u_{x}\neq \left\langle W\right\rangle _{s,t}\text{ or }%
\left\vert u_{\bot }\right\vert >\frac{2r}{t-s}\right\} , \\
\chi _{1}(t,\mathbf{u}) &=&\chi \left\{ \mathbf{u:\exists }s\in (0,t)\text{
s.t. }u_{x}=\left\langle W\right\rangle _{s,t}\text{ and }\left\vert u_{\bot
}\right\vert \leqslant \frac{2r}{t-s}\right\} .
\end{eqnarray*}
In case the precollisions occurred at a sequence of earlier times $%
t_{j}\rightarrow t$, it would follow that $v_{x}=W(t)$, so there would be no
contribution to the force since $\ell(0)=0$. Thus we can assume that there
is a first precollision, that is, a collision that occurs at an earlier time
closest to $t$. In that case let $\tau $ be the time and $\mathbf{\xi }$ be
the position of that first precollision. Of course, $\tau $ and $\mathbf{\xi 
}$ depend on $t,\mathbf{{x},{u}}$. We can then write 
\begin{equation}
f_{-}(t,\mathbf{x};\mathbf{u})=f_{+}(\tau ,\mathbf{\xi };\mathbf{u})\chi
_{1}(t,\mathbf{u})+f_{0}\left( \mathbf{u}\right) \chi _{0}(t,\mathbf{u}).
\label{equation:f- with precollision}
\end{equation}%
Plugging (\ref{equation:f- with precollision}) into the boundary condition,
we have 
\begin{eqnarray*}
f_{+}(t,\mathbf{x};\mathbf{v}) &=&\alpha f_{-}(t,\mathbf{x}%
;2W(t)-v_{x},v_{\bot })+(1-\alpha )\int_{u_{x}\leqslant W\left( t\right)
}K\left( \mathbf{v}-\mathbf{i}W\left( t\right) ;\mathbf{u}-\mathbf{i}W\left(
t\right) \right) f_{-}(t,\mathbf{x};\mathbf{u})d\mathbf{u} \\
&=&\alpha \left\{ f_{+}(\tau ,\mathbf{\xi };2W(t)-v_{x},v_{\bot })\chi
_{1}(t,2W(t)-v_{x},v_{\bot })+f_{0}\left( 2W(t)-v_{x},v_{\bot }\right) \chi
_{0}(t,2W(t)-v_{x},v_{\bot })\right\} \\
&&+(1-\alpha )\int_{u_{x}\leqslant W\left( t\right) }K\left( \mathbf{v}-%
\mathbf{i}W\left( t\right) ;\mathbf{u}-\mathbf{i}W\left( t\right) \right)
\left\{ f_{+}(\tau ,\mathbf{\xi };\mathbf{u})\chi _{1}(t,\mathbf{u}%
)+f_{0}\left( \mathbf{u}\right) \chi _{0}(t,\mathbf{u})\right\} d\mathbf{u.}
\end{eqnarray*}%
Since the momentum is only transferred horizontally, we can rewrite this
formula as 
\begin{eqnarray*}
&&a_{+}(t,\mathbf{x;}v_{x})b(v_{\bot }) \\
&=&\alpha \left\{ a_{+}(\tau ,\mathbf{\xi ;}2W(t)-v_{x})b(v_{\bot })\chi
_{1}(t,2W(t)-v_{x},v_{\bot })+a_{0}(2W(t)-v_{x})b(v_{\bot })\chi
_{0}(t,2W(t)-v_{x},v_{\bot })\right\} \\
&&+(1-\alpha )b(v_{\bot })\int_{u_{x}\leqslant W\left( t\right)
}k(v_{x}-W(t),u_{x}-W(t))\left\{ a_{+}(\tau ,\mathbf{\xi ;}u_{x})b(u_{\bot
})\chi _{1}(t,\mathbf{u})+a_{0}(u_{x})b(u_{\bot })\chi _{0}(t,\mathbf{u}%
)\right\} d\mathbf{u}.
\end{eqnarray*}
We do not divide by $b(v_{\bot })$ on both sides because it could possibly
vanish. Now 
\begin{eqnarray*}
a_{+}(t,\mathbf{x;}v_{x})b(v_{\bot }) &\leqslant &\alpha \left[ a_{+}(\tau ,%
\mathbf{\xi ;}2W(t)-v_{x})b(v_{\bot })+a_{0}(2W(t)-v_{x})b(v_{\bot })\right]
\\
&&+b(v_{\bot })\left\{ \sup_{\tau ,u_{x},\mathbf{\xi \in \partial \Omega
(\tau )}}a_{+}(\tau ,\mathbf{\xi ;}u_{x})\right\} \int_{\left\langle
W\right\rangle _{t}}^{W(t)}k(v_{x}-W(t),u_{x}-W(t))du_{x} \\
&&+b(v_{\bot })\int_{\mathbb{-\infty }}^{V_{\infty
}}k(v_{x}-W(t),u_{x}-W(t))a_{0}(u_{x})du_{x}
\end{eqnarray*}
by \eqref{velocitybound}. Since $W\left( t\right) -\left\langle
W\right\rangle _{t}\leqslant \gamma =V_{\infty }-V_{0}$, we have 
\begin{eqnarray*}
a_{+}(t,\mathbf{x;}v_{x})b(v_{\bot }) &\leqslant &\alpha \left[ a_{+}(\tau ,%
\mathbf{\xi ;}2W(t)-v_{x})b(v_{\bot })+a_{0}(2W(t)-v_{x})b(v_{\bot })\right]
\\
&&+b(v_{\bot })C\gamma \left\{ \sup_{\tau ,u_{x},\mathbf{\xi \in \partial
\Omega (\tau )}}a_{+}(\tau ,\mathbf{\xi ;}u_{x})\right\} +b(v_{\bot })C
\end{eqnarray*}
by A2 and A4. Hence, taking the supremum over all times $t$, positions $%
\mathbf{x}\in\partial\Omega(t)$ and velocities $v_x\in\mathbb{R}$, we have 
\begin{eqnarray*}
b(v_{\bot })\left\{ \sup_{\tau ,u_{x},\mathbf{\xi \in \partial \Omega (\tau )%
}}a_{+}(\tau ,\mathbf{\xi ;}u_{x})\right\} &\leqslant &\alpha b(v_{\bot
})\left\{ \sup_{\tau ,u_{x},\mathbf{\xi \in \partial \Omega (\tau )}%
}a_{+}(\tau ,\mathbf{\xi ;}u_{x})\right\} +Cb(v_{\bot }) \\
&&+b(v_{\bot })C\gamma \left\{ \sup_{\tau ,u_{x},\mathbf{\xi \in \partial
\Omega (\tau )}}a_{+}(\tau ,\mathbf{\xi ;}u_{x})\right\} +Cb(v_{\bot }).
\end{eqnarray*}%
That is, 
\begin{equation}
b(v_{\bot })\left\{ \sup_{\tau ,u_{x},\mathbf{\xi \in \partial \Omega (\tau )%
}}a_{+}(\tau ,\mathbf{\xi ;}u_{x})\right\} \leqslant \frac{Cb(v_{\bot })}{%
1-\alpha -C\gamma }\leqslant Cb(v_{\bot }),\text{ for }\gamma <\frac{%
1-\alpha }{C},  \label{bound:upper bound for f+}
\end{equation}%
which is an upper bound for $f_{+}(t,\mathbf{x};\mathbf{v}).$

In order to get a lower bound of $f_{+}(t,\mathbf{x};\mathbf{v}),$ we use
Assumption A5 to deduce, for $V_{0}\leq v_{x}\leq V_{\infty }$, that 
\begin{eqnarray}
f_{+}(t,\mathbf{x};\mathbf{v}) &=&a_{+}(t,\mathbf{x;}v_{x})b(v_{\bot })
\label{bound:lower bound of IL} \\
&\geqslant &(1-\alpha )b(v_{\bot })\int_{u_{x}\leqslant W\left( t\right)
}k(v_{x}-W(t),u_{x}-W(t))a_{0}(u_{x})b(u_{\bot })\chi _{0}(t,\mathbf{u})d%
\mathbf{u}  \notag \\
&\geqslant &(1-\alpha )b(v_{\bot })\int_{-\infty
}^{V_{0}}k(v_{x}-W(t),u_{x}-W(t))a_{0}(u_{x})du_{x}  \notag \\
&\geqslant &(a_{0}(V_{\infty })+\delta )\ b(v_{\perp })  \notag
\end{eqnarray}
by A5.

We are now ready to establish upper and lower bounds of $r_{W}^{R}\left(
t\right) .$ We begin with the crucial lower bound because it is the main
reason why $r_{W}^{R}+r_{W}^{L}\geqslant 0$. Using the lower bound of $%
f_{+}\left( \tau ,\mathbf{\xi },\mathbf{u}\right) $, we get 
\begin{eqnarray*}
r_{W}^{R}\left( t\right) &=&\int_{\partial \Omega _{R}\left( t\right) }dS_{%
\mathbf{x}}\int_{u_{x}\leqslant W\left( t\right) }d\mathbf{u}\
\ell(u_{x}-W(t))\left\{ f_{-}(t,\mathbf{x},\mathbf{u})-f_{0}(\mathbf{u}%
)\right\} \\
&=& \int_{\partial \Omega _{R}\left( t\right) }dS_{\mathbf{x}%
}\int_{u_{x}\leqslant W\left( t\right) }d\mathbf{u}\
\ell(u_{x}-W(t))[f_{+}\left( \tau ,\mathbf{\xi },\mathbf{u}\right) \chi
_{1}\left( t,\mathbf{u}\right) +f_{0}(\mathbf{u})\chi _{0}\left( t,\mathbf{u}%
\right) -f_{0}(\mathbf{u})] \\
&=& \int_{\partial \Omega _{R}\left( t\right) }dS_{\mathbf{x}%
}\int_{\left\vert u_{\bot }\right\vert \leqslant \frac{2r}{t-\tau }}du_{\bot
}\int_{\left\langle W\right\rangle _{t}}^{W\left( t\right) }du_{x}\
\ell(u_{x}-W(t))(f_{+}\left( \tau ,\mathbf{\xi },\mathbf{u}\right) -f_{0}(%
\mathbf{u})) \\
&\geqslant & \int_{\partial \Omega _{R}\left( t\right) }dS_{\mathbf{x}
}\int_{\left\vert u_{\bot }\right\vert \leqslant \frac{2r}{t-\tau }%
}du_{\bot} b(u_\perp) \int_{\left\langle W\right\rangle _{t}}^{W\left(
t\right) }du\ \ell(u_{x}-W(t)) \left( a_{0}(V_{\infty })+\delta
-a_{0}(u_{x})\right)
\end{eqnarray*}%
by (\ref{bound:lower bound of IL}).  For small enough $\gamma $, we have by
continuity of $a_{0}$ that%
\begin{equation*}
a_{0}(V_{\infty })+\delta -a_{0}(u_{x})\geqslant \frac{\delta }{2}>0.
\end{equation*}
We then deduce via Assumption A3 and Lemma \ref{Lemma:The class of W}(iv)
that 
\begin{equation*}
r_{W}^{R}\left( t\right) \geqslant C\frac{\delta }{2}\int_{\left\vert
u_{\bot }\right\vert \leqslant \frac{2r}{t-\tau }} b(u_\perp)
du_{\bot}\int_{\left\langle W\right\rangle _{t}}^{W\left( t\right)
}du_{x}\ell(u_{x}-W(t))\geqslant C\frac{\delta }{2}\frac{\left(
W(t)-\left\langle W\right\rangle _{t}\right) ^{p+1}}{\left( 1+t\right) ^{d-1}%
}\geq 0.
\end{equation*}%
Thus 
\begin{equation*}
r_{W}^{R}\left( t\right) \geqslant \frac{C\delta }{\left( 1+t\right) ^{d-1}}%
\gamma ^{p+1}\left( \langle h\rangle _{t}-g(t)\right) ^{p+1}
\end{equation*}%
for $t\geqslant t_{0}$ and small enough $\gamma $. This is the desired the
lower bound of $r_{W}^{R}$.

We now determine an upper bound for $r_{W}^{R}.$ Using the upper bound %
\eqref{bound:upper bound for f+} of $f_{+}\left( \tau ,\mathbf{\xi },\mathbf{%
u}\right) $ and Lemma \ref{Lemma:The class of W}(iv), we have 
\begin{eqnarray*}
r_{W}^{R}\left( t\right) &=&\int_{\partial \Omega _{R}\left( t\right) }dS_{%
\mathbf{x}}\int_{u_{x}\leqslant W\left( t\right) }d\mathbf{u}\
\ell(u_{x}-W(t))\left\{ f_{+}\left( \tau ,\mathbf{\xi },\mathbf{u}\right)
-f_{0}(\mathbf{u})\right\} \\
&=&\int_{\partial \Omega _{R}\left( t\right) }dS_{\mathbf{x}%
}\int_{\left\langle W\right\rangle _{t}}^{W\left( t\right)
}du_{x}\int_{\left\vert u_{\bot }\right\vert \leqslant \frac{2r}{t-\tau }%
}du_{\bot }\ell(u_{x}-W(t))(f_{+}\left( \tau ,\mathbf{\xi },\mathbf{u}%
\right) -f_{0}(\mathbf{u})) \\
&\leqslant &C\int_{\left\langle W\right\rangle _{t}}^{W\left( t\right)
}du_{x}\int_{\left\vert u_{\bot }\right\vert \leqslant \frac{2r}{t-\tau }%
}du_{\bot }\ell(u_{x}-W(t))\ b(u_{\bot }).
\end{eqnarray*}%
We split the integral according to whether $\tau <t/2$ or $\tau \geq t/2$.
Thus 
\begin{eqnarray*}
r_{W}^{R}\left( t\right) &\leq &\frac{C\left( W(t)-\left\langle
W\right\rangle _{t}\right) ^{p+1}}{\left( 1+t\right) ^{d-1}}%
+C\int_{\left\langle W\right\rangle _{t}}^{W\left( t\right)
}du_{x}\int_{\left\vert u_{\bot }\right\vert \leqslant \frac{2r}{t-\tau }%
\text{,}\tau \geqslant \frac{t}{2}}du_{\bot }\ell(u_{x}-W(t))b(u_{\bot }) \\
&\leqslant &\frac{C\left( \frac{\gamma }{t}\int_{0}^{t}g(\tau )d\tau \right)
^{p+1}}{\left( 1+t\right) ^{d-1}}+C\int_{\left\langle W\right\rangle
_{t}}^{W\left( t\right) }du_{x}\int_{\left\vert u_{\bot }\right\vert
\leqslant \frac{2r}{t-\tau }\text{,}\tau \geqslant \frac{t}{2}}du_{\bot
}\ell(u_{x}-W(t))\ b(u_{\bot })
\end{eqnarray*}%
by Assumption A3. For the second term in this estimate, by the precollision
condition (\ref{condition:recollision condition}) we notice that%
\begin{equation*}
u_{x}=W(t)-\frac{1}{t-\tau }\int_{\tau }^{t}\left( W(t)-W(s)\right)
ds\geqslant W(t)-\frac{1}{t-\tau }\int_{\tau }^{t}\left( V_{\infty
}-W(s)\right) ds\geqslant W(t)-\gamma \langle g\rangle _{\tau ,t}.
\end{equation*}%
By Assumption A3 again, this inequality allows us to estimate the second
term as 
\begin{eqnarray*}
&&\int_{W(t)-\langle g\rangle _{\tau ,t}}^{W\left( t\right)
}du_{x}\int_{\left\vert u_{\bot }\right\vert \leqslant \frac{2r}{t-\tau }%
\text{,}\tau \geqslant \frac{t}{2}}du_{\bot }\ell(u_{x}-W(t))b(u_{\bot }) \\
&\leqslant &C\sup_{\frac{t}{2}\leqslant \tau \leqslant t}\int_{-\gamma
\langle g\rangle _{\tau ,t}}^{0}du_{x}\left\vert u_{x}\right\vert
^{p}\leqslant C\gamma ^{p+1}\sup_{\frac{t}{2}\leqslant \tau \leqslant
t}\langle g\rangle _{\tau ,t}^{p+1}\ .
\end{eqnarray*}%
Hence%
\begin{equation*}
r_{W}^{R}\left( t\right) \leqslant \frac{C\gamma ^{p+1}\langle g\rangle
_{t}^{p+1}}{\left( 1+t\right) ^{d-1}}+C\gamma ^{p+1}\sup_{\frac{t}{2}%
\leqslant \tau \leqslant t}\langle g\rangle _{\tau ,t}^{p+1}.
\end{equation*}
\end{proof}


\subsection{The Left Side}

We now proceed to bound the force $\left\vert r_{W}^{L}\right\vert $ on the
left side of the cylinder.

\begin{lemma}
\label{Lemma:UpperBoundOfR-}Under the same assumptions as in Lemma \ref%
{Lemma:Upper and lower bound of R+}, for $\gamma $ small enough, we have 
\begin{equation*}
\left\vert r_{W}^{L}\left( t\right) \right\vert \leqslant C\gamma ^{p+1}\chi
\left\{ t\geqslant t_{0}\right\} \left( \frac{g^{p+1}(t)}{t^{d-1}}+\sup_{%
\frac{t}{2}\leqslant \tau \leqslant t}\langle g\rangle _{\tau
,t}^{p+1}\right) .
\end{equation*}
\end{lemma}

This estimate is different from the upper bound of $r_{W}^{R}\left( t\right) 
$ because it is the second term that is the dominant one.

\begin{proof}
We first notice that $r_{W}^{L}\left( t\right) =0$ for all $t\leqslant t_{0}$
because $W$ is increasing. Indeed, suppose that on the left there is a
precollision at time $\tau $ and a later collision at time $t\le t_0$. If
the velocity of the particle in the time period $\left( \tau ,t\right) $ is $%
\mathbf{u},$ then $u_{x}\leqslant W\left( \tau \right) $ and $u_{x}\geqslant
W(t)>W(\tau )$ which is a contradiction.

Now by the precollision condition, we have%
\begin{equation*}
u_{x}=\left\langle W\right\rangle _{\tau ,t}=\frac{1}{t-\tau }\int_{\tau
}^{t}W\left( s\right) ds\leqslant \frac{1}{t-\tau }\int_{\tau }^{t}\left(
V_{\infty }-\gamma h(s)\right) ds\leqslant V_{\infty }.
\end{equation*}%
Recalling the boundary condition (\ref{boundary condition: Left}) on the
left side of the cylinder, we have 
\begin{equation*}
f_{+}(t,\mathbf{x};\mathbf{v})=\alpha f_{-}(t,\mathbf{x},2W(t)-v_{x},v_{\bot
})+(1-\alpha )\int_{u_{x}\geqslant W\left( t\right) }K\left( \mathbf{v}-%
\mathbf{i}W\left( t\right) ,\mathbf{u}-\mathbf{i}W\left( t\right) \right)
f_{-}(t,\mathbf{x};\mathbf{u})d\mathbf{u}.
\end{equation*}%
Again, plugging in the precollision condition (\ref{equation:f- with
precollision}), namely $f_{-}(t,\mathbf{x};\mathbf{u})=f_{+}(\tau ,\mathbf{%
\xi };\mathbf{u})\chi _{1}\left( t,\mathbf{u}\right) +f_{0}(\mathbf{u})\chi
_{0}(t,\mathbf{u}),$ we then have%
\begin{eqnarray*}
f_{+}(t,\mathbf{x};\mathbf{v}) &=&\alpha \left\{ f_{+}(\tau ,\mathbf{\xi }%
;2W(t)-v_{x},v_{\bot })\chi _{1}\left( t,2W(t)-v_{x},v_{\bot }\right)
+f_{0}(2W(t)-v_{x},v_{\bot })\chi _{0}(t,2W(t)-v_{x},v_{\bot })\right\} \\
&&+(1-\alpha )\int_{u_{x}\geqslant W\left( t\right) }K\left( \mathbf{v}-%
\mathbf{i}W\left( t\right) ,\mathbf{u}-\mathbf{i}W\left( t\right) \right)
\left\{ f_{+}(\tau ,\mathbf{\xi };\mathbf{u})\chi _{1}\left( t,\mathbf{u}%
\right) +f_{0}(\mathbf{u})\chi _{0}(t,\mathbf{u})\right\} d\mathbf{u.}
\end{eqnarray*}%
Together with $u_{x}\leqslant V_{\infty }$, we have 
\begin{eqnarray*}
&&f_{+}(t,\mathbf{x};\mathbf{v})=a_{+}(t\mathbf{,x;}v_{x})b(v_{\bot }) \\
&=&\alpha \left\{ a_{+}(\tau ,\mathbf{\xi };2W(t)-v_{x})\chi _{1}\left(
t,2W(t)-v_{x},v_{\bot }\right) +a_{0}(2W(t)-v_{x})\chi
_{0}(t,2W(t)-v_{x},v_{\bot })\right\} b(v_{\bot }) \\
&&+(1-\alpha )b(v_{\bot })\int_{u_{x}\geqslant W\left( t\right) }k\left(
v_{x}-W\left( t\right) ,u_{x}-W\left( t\right) \right) \left\{ a_{+}(\tau ,%
\mathbf{\xi };u_{x})b(u_{\bot })\chi _{1}\left( t,\mathbf{u}\right)
+a_{0}(u_{x})b(u_{\bot })\chi _{0}(t,\mathbf{u})\right\} d\mathbf{u} \\
&\leqslant &\alpha \left\{ a_{+}(\tau ,\mathbf{\xi };2W(t)-v_{x})\chi
_{1}\left( t,2W(t)-v_{x},v_{\bot }\right) +a_{0}(2W(t)-v_{x})\chi
_{0}(t,2W(t)-v_{x},v_{\bot })\right\} b(v_{\bot }) \\
&&+(1-\alpha )b(v_{\bot })\int_{W\left( t\right) }^{V_{\infty }}k\left(
v_{x}-W\left( t\right) ,u_{x}-W\left( t\right) \right) a_{+}(\tau ,\mathbf{%
\xi };u_{x})du_{x} \\
&&+(1-\alpha )b(v_{\bot })\int_{-\infty }^{V_{\infty }}k\left( v_{x}-W\left(
t\right) ,u_{x}-W\left( t\right) \right) a_{0}(u_{x})du_{x} \\
&\leqslant &\alpha \left\{ a_{+}(\tau ,\mathbf{\xi };2W(t)-v_{x})\chi
_{1}\left( t,2W(t)-v_{x},v_{\bot }\right) +a_{0}(2W(t)-v_{x})\chi
_{0}(t,2W(t)-v_{x},v_{\bot })\right\} b(v_{\bot }) \\
&&+Cb(v_{\bot })\gamma \left\{ \sup_{\tau ,u_{x},\mathbf{\xi \in \partial
\Omega (\tau )}}a_{+}(\tau ,\mathbf{\xi ;}u_{x})\right\} +C
\end{eqnarray*}%
for $v_{x}\in \lbrack V_{0},V_{\infty }]$, where in the last line we used
Assumptions A2 and A5. Hence, taking supremums as in the earlier estimate (%
\ref{bound:upper bound for f+}), we have 
\begin{equation*}
b(v_{\bot })\left\{ \sup_{\tau ,u_{x},\mathbf{\xi \in \partial \Omega (\tau )%
}}a_{+}(\tau ,\mathbf{\xi ;}u_{x})\right\} \leqslant b(v_{\bot })\{\alpha
+C\gamma \}\left\{ \sup_{\tau ,u_{x},\mathbf{\xi \in \partial \Omega (\tau )}%
}a_{+}(\tau ,\mathbf{\xi ;}u_{x})\right\} +Cb(v_{\bot }).
\end{equation*}%
Since $\alpha <1$ and $\gamma $ is small, we deduce that 
\begin{equation*}
b(v_{\bot })\left\{ \sup_{\tau ,u_{x},\mathbf{\xi \in \partial \Omega (\tau )%
}}a_{+}(\tau ,\mathbf{\xi ;}u_{x})\right\} \leqslant Cb(v_{\bot }).
\end{equation*}%
Using this upper bound of $f_{+}(\tau ,\mathbf{\xi },u)$, we get 
\begin{eqnarray*}
\left\vert r_{W}^{L}\left( t\right) \right\vert &=&\left\vert \int_{\partial
\Omega _{L}(t)}dS_{x}\int_{u\geqslant W\left( t\right) }d\mathbf{u}\
\ell(u_{x}-W(t))\left\{ f_{-}(t,\mathbf{x};\mathbf{u})-f_{0}(\mathbf{u}%
)\right\} \right\vert \\
&\leqslant &\int_{\partial \Omega _{L}(t)}dS_{x}\int_{u\geqslant W\left(
t\right) }d\mathbf{u}\ \ell(u_{x}-W(t))\left\vert f_{+}(\tau ,\mathbf{\xi };%
\mathbf{u})\chi _{1}\left( t,\mathbf{u}\right) +f_{0}(\mathbf{u})\chi _{0}(t,%
\mathbf{u})-f_{0}(\mathbf{u})\right\vert \\
&\leqslant &C\int_{\left\vert u_{\bot }\right\vert \leqslant \frac{2r}{%
t-\tau }}du_{\bot }\int_{W(t)}^{V_{\infty }}du_{x}\
\ell(u_{x}-W(t))b(u_{\bot })
\end{eqnarray*}%
As before, we split the integral at $\tau =t/2$. So, as in the proof of the
previous lemma, we obtain 
\begin{eqnarray*}
\left\vert r_{W}^{L}\left( t\right) \right\vert &\leqslant &\frac{C\left(
V_{\infty }-W(t)\right) ^{p+1}}{\left( 1+t\right) ^{d-1}}\chi {\{t\geq
t_{0}\}}+C\sup_{\frac{t}{2}\leqslant \tau \leqslant t}\left( \frac{\gamma }{%
t-\tau }\int_{\tau }^{t}g(s)ds\right) ^{p+1} \\
&\leqslant &\frac{C\gamma ^{p+1}g^{p+1}(t)}{\left( 1+t\right) ^{d-1}}\chi {%
\{t\geq t_{0}\}}+C\gamma ^{p+1}\sup_{\frac{t}{2}\leqslant \tau \leqslant
t}\langle g\rangle _{\tau ,t}^{p+1}
\end{eqnarray*}%
for small $\gamma $, by Assumption A3.
\end{proof}


\subsection{Force Due to Precollisions}

\begin{lemma}
\label{Corollary:UpperAndLowerBoundOfR}Define 
\begin{equation*}
R_W(t)=r_{W}^{R}(t)+r_{W}^{L}(t).
\end{equation*}%
Assume that $g$ is non-increasing and that there is a power $M>\frac{p+d}{p+1%
}$ and a constant $G$ such that 
\begin{equation}
g(t)\leq G(1+t)^{-M}.  \label{Gbound}
\end{equation}%
Then 
\begin{equation}
R_W(t)\leq C\gamma ^{p+1}\frac{G^{p+1}}{(1+t)^{d+p}}.  \label{upperR}
\end{equation}
Furthermore, for $t\geq \max ((2G/H)^{1/\left( M-1\right) },4^q\left(
2^{M}G/H\right) ^{q(p+1)}$ , we have 
\begin{equation}
R_W(t)\geq C\gamma ^{p+1}\frac{\chi \{t\geq t_{0}\}}{t^{d+p}}H^{p+1},
\label{lowerR}
\end{equation}%
where $H=\int_{0}^{1}h(s)ds$ and $q = \{M(p+1)-(p+d)\}^{-1}$. Here and
below, the constant $C$ may change from line to line but is always
independent of $t,\gamma ,H,G,g(0)$.
\end{lemma}

\begin{proof}
By the monotonicity, $\langle g\rangle _{t/2,t}\leq g(t/2)$. So by Lemmas %
\ref{Lemma:Upper and lower bound of R+} and \ref{Lemma:UpperBoundOfR-}, we
have 
\begin{equation*}
R_W(t)\leq C\gamma ^{p+1}\left\{ \frac{1}{(1+t)^{d-1}}\langle g\rangle
_{t}^{p+1}+g^{p+1}(\frac{t}{2})\right\} .
\end{equation*}%
Thus by (\ref{Gbound}) we have 
\begin{equation*}
R_W(t)\leq C\gamma ^{p+1}\frac{G^{p+1}}{(1+t)^{d+p}}.
\end{equation*}%
which is the desired upper bound. Next, by Lemmas \ref{Lemma:Upper and lower
bound of R+} and \ref{Lemma:UpperBoundOfR-} we have the lower bound 
\begin{equation*}
R_W(t)\geq \chi \{t\geq t_{0}\}C\gamma ^{p+1}\left\{ \frac{[\langle h\rangle
_{t}-g(t)]^{p+1}}{t^{d-1}}-g^{p+1}(t/2)\right\}
\end{equation*}%
\begin{equation*}
\geq \chi \{t\geq t_{0}\}C\gamma ^{p+1}\left\{ \left[ \frac{H}{t}-\frac{G}{%
t^{{M}}}\right] ^{p+1}{t^{1-d}}-\frac{G^{p+1}}{(t/2)^{M(p+1)}}\right\} \geq
\chi \{t\geq t_{0}\}C\gamma ^{p+1}\left\{ \frac{H^{p+1}}{2t^{p+d}}%
-G^{p+1}\left( \frac{2}{t}\right) ^{M(p+1)}\right\}
\end{equation*}%
provided $t>(2G/H)^{1/\left( M-1\right) }$. Therefore 
\begin{equation*}
R_W(t)\geq \chi \{t\geq t_{0}\}C\gamma ^{p+1}\frac{H^{p+1}}{4t^{d+p}}
\end{equation*}%
provided also that $t> 4^q\left( 2^{M}G/H\right) ^{q(p+1)}$. 
\end{proof}


\section{Motion of the Body}

Combining Lemmas \ref{Lemma:Upper and lower bound of R+}, \ref%
{Lemma:UpperBoundOfR-} and \ref{Corollary:UpperAndLowerBoundOfR}, we can now
determine upper and lower bounds of $V_\infty-V_W(t)$.

\begin{lemma}
\label{Prop:DeducingConditionsOnhAndg} Define the quotient%
\begin{equation*}
Q(t)=\frac{F_{0}(V_{\infty })-F_{0}(W(t))}{V_{\infty }-W(t)},
\end{equation*}%
the two positive constants 
\begin{equation*}
B_{0}=\min_{V\in \left[ V_{0},V_{\infty }\right] }F_{0}^{\prime }(V),\quad
B_{\infty }=\max_{V\in \left[ V_{0},V_{\infty }\right] }F_{0}^{\prime }(V)
\end{equation*}%
and the cutoff time 
\begin{equation*}
t_{0}=\frac{1}{2B_{\infty }}\log \frac{B_{0}}{\gamma ^{p}}.
\end{equation*}%
Assuming \eqref{Gbound} and that $\gamma $ is small enough, we then have the
following conclusions.

\begin{description}
\item[(i)] As a function of $t,$ $V_{W}(t)$ is differentiable with bounded
derivatives and it is increasing over the interval $[0,t_{0}]$.

\item[(ii)] For $t\geqslant 0,$ we have the upper bound 
\begin{equation*}
V_{\infty }-V_{W}(t)\leqslant \gamma e^{-B_{0}t}+C\gamma ^{p+1}
G^{p+1}(1+t)^{-d-p} .
\end{equation*}

\item[(iii)] For $t\geqslant 0,$ we have the lower bound 
\begin{equation*}
V_{\infty }-V_{W}(t)\geqslant \gamma e^{-B_{\infty }t}+CH^{p+1}\gamma
^{p+1}t^{-d-p}\chi _{(2t_{0},\infty )}\left( t\right) .
\end{equation*}
\end{description}
\end{lemma}

\begin{proof}
(i) By \eqref{iteration equation} we have 
\begin{equation}
\frac{d}{dt}(V_{\infty }-V_{W}(t))=-Q(t)(V_{\infty }-V_{W}(t))+R_W(t),
\label{equation:diff(V_inf-V_w)}
\end{equation}%
so that 
\begin{equation}
V_{\infty }-V_{W}(t)=\gamma e^{-\int_{0}^{t}Q\left( s\right)
ds}+\int_{0}^{t}e^{-\int_{s}^{t}Q\left( \tau \right) d\tau }R_W(s)ds
\label{equation:V_inf-V_w}
\end{equation}%
since $V_{\infty }-V_{W}(0)=V_{\infty }-V_{0}=\gamma $. By Lemma \ref%
{Lemma:The class of W} we have $V_{0}\leqslant V_{\infty }-\gamma g\leqslant
W(t)\leqslant V_{\infty }-\gamma h\leqslant V_{\infty }$. By Lemma \ref%
{fictitious lemma}, 
\begin{equation*}
0<B_{0}=\min_{V\in \left[ V_{0},V_{\infty }\right] }F_{0}^{\prime
}(V)\leqslant Q(t)=\frac{1}{V_{\infty }-W(t)}\int_{W(t)}^{V_{\infty
}}F_{0}^{\prime }(s)ds\leqslant \max_{V\in \left[ V_{0},V_{\infty }\right]
}F_{0}^{\prime }(V)=B_{\infty } <\infty.
\end{equation*}%
Thus (\ref{equation:V_inf-V_w}) together with the positivity of $R_W$
implies the lower bound 
\begin{equation*}
V_{\infty }-V_{W}(t)\geqslant \gamma e^{-\int_{0}^{t}Q}\geqslant \gamma
e^{-B_{\infty }t}.
\end{equation*}%
Now (\ref{equation:diff(V_inf-V_w)}) gives us the upper bound 
\begin{equation*}
\frac{d}{dt}\left( V_{\infty }-V_{W}(t)\right) \leqslant -\min \left[ Q(t)%
\right] \cdot \min \left[ V_{\infty }-V_{W}(t)\right] +R_W(t)\leqslant
-B_{0}\gamma e^{-B_{\infty }t}+R_W(t).
\end{equation*}%
Applying the upper bound on $R_W(t)$, we then turn this estimate into 
\begin{equation*}
\frac{d}{dt}\left( V_{\infty }-V_{W}(t)\right) \leqslant -B_{0}\gamma
e^{-B_{\infty }t}+C\gamma ^{p+1}<0
\end{equation*}%
for $0\leqslant t\leqslant t_{0}$, provided $({B_{0}}/{\gamma }%
^{p})e^{-B_{\infty }t_{0}}>C$. We choose $t_{0}$ as above so that $%
t_{0}\rightarrow +\infty $ as $\gamma \rightarrow 0$, and $({B_{0}}/{\gamma }%
^{p})e^{-B_{\infty }t_{0}}=e^{+B_{\infty }t_{0}}>>0$.

(ii) By (\ref{equation:V_inf-V_w}), we have the upper bound 
\begin{equation*}
V_{\infty }-V_{W}(t)\leqslant \gamma
e^{-B_{0}t}+\int_{0}^{t}e^{-B_{0}(t-s)}R_W(s)\ ds.
\end{equation*}%
We split the integral into two parts. On the one hand, the integral from $%
t/2 $ to $t$ is bounded above by 
\begin{equation*}
\int_{\frac{t}{2}}^{t}e^{-B_{0}(t-s)}R_W(s)ds\leq \int_{\frac{t}{2}%
}^{t}e^{-B_{0}(t-s)}C\gamma ^{p+1}\frac{G^{p+1}}{(1+s)^{d+p}}ds\leq C\gamma
^{p+1}\frac{G^{p+1}}{(1+t)^{d+p}}.
\end{equation*}%
On the other hand, the integral from $0$ to $t/2$ is bounded more simply by 
\begin{equation*}
\int_{0}^{t/2}e^{-B_{0}(t-s)}C\gamma ^{p+1}ds\leq \frac{C\gamma ^{p+1}}{B_{0}%
}e^{-B_{0}t/2}.
\end{equation*}%
Thus 
\begin{equation}
V_{\infty }-V_{W}(t)\leqslant \gamma e^{-B_{0}t}+C\gamma ^{p+1}\frac{G^{p+1}%
}{(1+t)^{d+p}}.  \label{UPPER}
\end{equation}

(iii) On the other hand, by (\ref{equation:V_inf-V_w}) we have the lower
bound 
\begin{eqnarray*}
V_{\infty }-V_{W}(t) &\geqslant &\gamma e^{-B_{\infty
}t}+\int_{0}^{t}e^{-B_{\infty }(t-s)}R_W(s)ds\geq \gamma e^{-B_{\infty
}t}+CH^{p+1}\gamma ^{p+1}\int_{t_{0}}^{t}e^{-B_{\infty }(t-s)}s^{-d-p}ds \\
&\geqslant &\gamma e^{-B_{\infty }t}+CH^{p+1}\gamma ^{p+1}\frac{%
1-e^{-B_{\infty }(t-t_{0})}}{B_{\infty }}t^{-d-p}
\end{eqnarray*}%
by \eqref{lowerR}. Now for $t\geq 2t_{0}$, we have $1-e^{-B_{\infty
}(t-t_{0})}\geq 1-e^{-B_{\infty }t_{0}}>\frac{1}{2}$ for large $t_{0}$
(small $\gamma $). Thus we have the desired lower bound 
\begin{equation}
V_{\infty }-V_{W}(t)\geqslant \gamma e^{-B_{\infty }t}+CH^{p+1}\gamma
^{p+1}t^{-d-p}\chi _{(2t_{0},\infty )}\left( t\right)  \label{LOWER}
\end{equation}%
with a different constant $C$.
\end{proof}


By (\ref{g and h}) we summarize the requirements on $g$ and $h$ as follows.

\begin{corollary}
\label{requirements} $V_{W}(t)\in \mathcal{W}$ provided the following
conditions are satisfied. 
\begin{eqnarray}
\gamma e^{-B_{0}t}+C\gamma ^{p+1}\frac{G^{p+1}}{(1+t)^{d+p}} &<&\gamma g(t),
\label{upper} \\
\gamma e^{-B_{\infty }t}+CH^{p+1}\gamma ^{p+1}t^{-d-p}\chi _{(2t_{0},\infty
)}\left( t\right) &>&\gamma h(t),  \label{lower}
\end{eqnarray}%
\begin{equation}
2t_{0}\geq \max \{ (2G/H)^{1/\left( M-1\right) },4^q \left( 2^{M}G/H\right)
^{q(p+1)} \} ,  \label{bound on t_0}
\end{equation}%
\begin{equation}
g(t)\leq G(1+t)^{-M}\text{ with }M>\frac{p+d}{p+1}  \label{G}
\end{equation}
\end{corollary}

\begin{corollary}
\label{ghSpecific} \label{Corollary:Choosing g and h} One can choose
constants $A_{+}$ and $A_{-}$ so that the pair 
\begin{eqnarray*}
g(t) &=&e^{-B_{0}t}+\frac{\gamma ^{p}A_{+}}{\left( 1+t\right) ^{d+p}} \\
h(t) &=&e^{-B_{\infty }t}+\frac{\gamma ^{p}A_{-}}{t^{d+p}}\chi
_{(2t_{0},\infty )}\left( t\right)
\end{eqnarray*}%
satisfies the conditions of Corollary \ref{requirements}.
\end{corollary}

\begin{proof}[Proof of Corollary \protect\ref{Corollary:Choosing g and h}]
Notice for (\ref{G}) we can take $G=1+\gamma A_{+}$. For (\ref{upper}) we
require 
\begin{equation*}
C\gamma ^{p}\frac{G^{p+1}}{(1+t)^{d+p}}<\frac{\gamma ^{p}A_{+}}{\left(
1+t\right) ^{d+p}},
\end{equation*}%
which is true provided $A_{+}>2C$ and $\gamma $ is sufficiently small.
Notice that $H=\int_{0}^{1}h(s)ds=\int_{0}^{1}\exp {(-B_{\infty }s)}%
ds=(1-\exp (-B_{\infty })/B_{\infty }>0$. For (\ref{lower}) we require 
\begin{equation*}
\frac{CH^{p+1}\gamma ^{p}}{t^{d+p}}\chi _{(2t_{0},\infty )}\left( t\right) >%
\frac{\gamma ^{p}A_{-}}{t^{d+p}}\chi _{(2t_{0},\infty )}(t),
\end{equation*}%
which merely requires $A_{-}<CH^{p+1}$. Finally, (\ref{bound on t_0}) is
true for small $\gamma $ because $t_{0}\rightarrow \infty $ as $\gamma
\rightarrow 0$.
\end{proof}


\section{Proof of Existence and Asymptotic Behavior}

\begin{proof}[Proof of Theorem \protect\ref{ThExistence}]
As in Definition \ref{def:W}, $\mathcal{W}$ is defined as the set of all
Lipschitz functions $W(t)$ on the half line $0\leq t<\infty $,
non-decreasing in $[0,t_{0}]$, such that 
\begin{equation*}
W(0)=V_{0},\quad \lim_{t\rightarrow \infty }W(t)=V_{\infty },\quad 0<\gamma
h(t)\leq V_{\infty }-W(t)<\gamma g(t).
\end{equation*}%
We define the \textquotedblleft ball" of radius $L$ in $\mathcal{W}$ as 
\begin{equation*}
\mathcal{K}=\{W\in \mathcal{W}\ \Big |\text{esssup}_{0\leqslant t<\infty
}(|W(t)|+|\dot{W}(t)|)\leq L\}
\end{equation*}%
for any positive constant $L$. Define $C_{b}([0,\infty ))$ to be the space
of continuous bounded functions on $[0,\infty )$. Of course, $\mathcal{K}$
is a compact and convex subset of $C_{b}([0,\infty ))$.

Given $W\in \mathcal{K}$, recall that $V_{W}$ is defined as the solution of
the differential equation 
\begin{equation}
\dot{V}_{W}=E-F_{0}(V_{W})-R_{W}(t),\quad V_{W}(t)=V_{0}.  \label{diffeqn}
\end{equation}%
Keeping in mind that $\left\vert R_{W}(t)\right\vert \leqslant C\gamma ^{p+1}
$ according to Lemma \ref{Corollary:UpperAndLowerBoundOfR}, we choose 
\begin{equation*}
L=\max \{V_{\infty },E+F_{0}(V_{\infty })+C\gamma ^{p+1}\}.
\end{equation*}%
We then consider the mapping $\mathcal{A}:W\rightarrow V_{W}$. By Corollary %
\ref{requirements}, $\mathcal{A}$ maps $\mathcal{K}$ into $\mathcal{K}$. By
Lemma \ref{Lemma:continuity of the map} below, $\mathcal{A}$ is continuous
in the topology of $C_{b}([0,\infty ))$, that is, with respect to uniform
convergence. By the Schauder fixed point theorem, $\mathcal{A}$ has a fixed
point in $\mathcal{K}$, which is our desired solution. Hence we have
concluded the proof of Theorem \ref{ThExistence}.
\end{proof}

\begin{lemma}
\label{Lemma:continuity of the map}The mapping $\mathcal{A}$ is continuous
in the topology of $C_{b}([0,\infty ))$.
\end{lemma}

\begin{proof}
Let $W_{j}\rightarrow W$ in $C_{b}([0,\infty ))$ where each $W_{j}\in 
\mathcal{K}$. By \eqref{diffeqn} it suffices to prove that $%
R_{W_{j}}(t)\rightarrow R_{W}(t)$ uniformly in $[0,\infty )$. Fix any $T>0$.
For any ${j}$ and any $N$, we define $B_{W_{j}}^{N}=\{(x,v_{x}):\exists \ $%
trajectory passing through $(T,x,v_{x})$ which has collided at least $N+1$
times in $[0,T]\}.$ Its complement is $A_{W_{j}}^{N}=\{(x,v_{x})$: no
trajectory passing through $(T,x,v_{x})$ has collided more than $N$ times in 
$[0,T]\}.$ We can then write $R_{W_{j}}\left( t\right) $ as a sum of
contributions from $A_{W_{j}}^{N}$ and $B_{W_{j}}^{N}$, namely, 
\begin{equation*}
R_{W_{j}}\left( t\right) =R_{W_{j}}\left( t;A_{W_{j}}^{N}\right)
+R_{W_{j}}\left( t;B_{W_{j}}^{N}\right) .
\end{equation*}%
Taking account of only the first precollision of each particle, we proved in
Lemmas \ref{Lemma:Upper and lower bound of R+} and \ref{Lemma:UpperBoundOfR-}
that 
\begin{equation}
0\leqslant R_{W}\left( t\right) \leqslant \frac{C\gamma ^{p+1}}{\left(
1+t\right) ^{p+d}}  \label{estimate:R estimate 1 in continuity}
\end{equation}%
with $C$ independent of both $t\in\mathbb{R}$ and $W\in\mathcal{W}$.
Iterating the same argument $N$ times, we have 
\begin{equation}
\sup_{0\leqslant t<\infty }\left\vert R_{W}\left( t;B_{W}^{N}\right)
\right\vert \leqslant \left( C\gamma ^{p+1}\right) ^{N}.
\label{estimate:R estimate 2 in continuity}
\end{equation}%
Thus it is natural to choose $\gamma <C^{-\frac{1}{p+1}}.$

Now we may write 
\begin{eqnarray*}
\sup_{0\leqslant t<\infty }\left\vert R_{W_{j}}\left( t\right) -R_{W}\left(
t\right) \right\vert &\leqslant &\sup_{T\leqslant t<\infty }\left\vert
R_{W_{j}}\left( t\right) -R_{W}\left( t\right) \right\vert +\sup_{0\leqslant
t<T}\left\vert R_{W}\left( t;B_{W}^{N}\right) \right\vert +\sup_{0\leqslant
t<T}\left\vert R_{W_{j}}\left( t;B_{W_{j}}^{N}\right) \right\vert \\
&&+\sup_{0\leqslant t<T}\left\vert R_{W_{j}}\left( t;A_{W_{j}}^{N}\right)
-R_{W}\left( t;A_{W}^{N}\right) \right\vert \\
&=&I+II+III+IV.
\end{eqnarray*}%
Let $\varepsilon >0$. By estimate (\ref{estimate:R estimate 1 in continuity}%
), we may choose $T=T_{\varepsilon }$ so large that $\left\vert I\right\vert
<\varepsilon /3.$ By estimate (\ref{estimate:R estimate 2 in continuity}),
we may then choose $N=N_{\varepsilon }$ so large that%
\begin{equation*}
\left\vert II\right\vert +\left\vert III\right\vert \leqslant 2\left(
C\gamma ^{p+1}\right) ^{N+1}<\varepsilon /3.
\end{equation*}%
In $IV$, there are no more than $N$ collisions. Therefore we can express
both terms in $IV$ as iterates of $N$ integrals by repeated use of the
collision boundary condition. The resulting finite number of iterated
integrals contain $W_{j}$ in a finite number of places. Therefore they
converge as $j\rightarrow \infty $ to the same expression with $W_{j}$
replaced by $W$. Thus we can choose $j$ so large that $\left\vert
IV\right\vert <\varepsilon /3$. Therefore $R_{W_{j}}(t)\rightarrow R_{W}(t)$
in $C_{b}([0,\infty ))$. Hence $\mathcal{A}$ is continuous in the topology
of $C_{b}([0,\infty ))$.
\end{proof}

\begin{proof}[Proof of Theorem \protect\ref{ThEvery}]
Let $\left( V(t),f(t,\mathbf{x},\mathbf{v})\right) $ be a solution of the
problem in the sense of Theorem \ref{ThExistence}. Then $V$ is a fixed point
of $\mathcal{A}$ so that $V\in\mathcal{W}$ and Corollary \ref{requirements}
is valid for it. Letting $g(t)$ and $h(t)$ be given by Corollary \ref%
{ghSpecific}, we then have 
\begin{equation*}
\gamma h(0)=\gamma =V_{\infty }-V(0)<\gamma g(0),
\end{equation*}%
so that $V_{\infty }-V(t)<\gamma g(t)$ for small enough $t$. Furthermore, 
\begin{equation*}
\frac{dV}{dt}\Big |_{t=0}=F_{0}(V_{\infty })-F(V_{0})=F_{0}(V_{\infty
})-F_{0}(V_{\infty }-\gamma )<\gamma \max_{V\in \left[ V_{0},V_{\infty }%
\right] }F_{0}^{\prime }(V)=\gamma B_{\infty },
\end{equation*}%
so that $V_{\infty }-V(t)>\gamma e^{-B_{\infty }t}=\gamma h(t)$ at least for
small enough $t>0$. Let 
\begin{equation*}
T=\inf \{s\ \Big |\ \gamma h(t)<V_{\infty }-V(t)<\gamma g(t),\ \forall \
0<t<s\}\leq \infty .
\end{equation*}%
In the interval $(0,T)$ the inequalities \eqref{UPPER} and \eqref{LOWER} are
satisfied. If $T$ were finite, we would have 
\begin{equation*}
V_{\infty }-V(T) = \gamma h(T) \ \ \text{ or }\ \ V_{\infty }-V(T) =\gamma
g(T),
\end{equation*}%
contradicting Corollary \ref{requirements}. Hence $T=\infty $. This proves
Theorem \ref{ThEvery}.
\end{proof}


\end{document}